\documentclass[reqno,a4paper,12pt]{amsart}
\usepackage{amssymb,enumerate, a4wide}

\usepackage[latin9]{inputenc}
\usepackage{enumitem}

\newcommand{\rb}{\ensuremath{\mathrm{r}}}
\newcommand{\rB}{{\mathrm B}}
\newcommand{\rC}{{\mathrm C}}
\newcommand{\Cent}{\ensuremath{{\rm{C}}}}
\newcommand{\NNN}{\ensuremath{{\mathrm{N}}}}
\newcommand{\R}{\mathrm {R}}
\newcommand{\T}{\mathrm {T}}

\newcommand{\bB}{{\mathbf B}}
\newcommand{\bC}{{\mathbf C}}
\newcommand{\bG}{{\mathbf G}}
\newcommand{\bH}{{\mathbf H}}
\newcommand{\bL}{{\mathbf L}}
\newcommand{\bM}{{\mathbf M}}
\newcommand{\bN}{{\mathbf N}}
\newcommand{\bS}{{\mathbf S}}
\newcommand{\bT}{{\mathbf T}}
\newcommand{\bu}{{\mathbf u}}

\newcommand{\bZ}{{\mathbf Z}}

\newcommand{\CC}{{\mathbb{C}}}
\newcommand{\FF}{{\mathbb{F}}}
\newcommand{\GG}{{\mathbb{G}}}
\newcommand{\LL}{{\mathbb{L}}}
\newcommand{\QQ}{{\mathbb{Q}}}
\newcommand{\TT}{{\mathbb{T}}}
\newcommand{\ZZ}{{\mathbb{Z}}}

\newcommand{\tA}{\mathsf A}
\newcommand{\tB}{\mathsf B}
\newcommand{\tC}{\mathsf C}
\newcommand{\tD}{\mathsf D}
\newcommand{\tE}{\mathsf E}
\newcommand{\tF}{\mathsf F}
\newcommand{\tG}{\mathsf G}

\newcommand{\fF}{{\mathfrak F}}

\newcommand{\cE}{{\mathcal E}}

\newcommand{\cG}{{\mathcal G}}
\newcommand{\cH}{{\mathcal H}}
\newcommand{\cN}{{\mathcal N}}

\newcommand{\cY}{{\mathcal Y}}
\newcommand{\cP}{{\mathcal P}}

\newcommand{\Aut}{\operatorname{Aut}}

\newcommand{\End}{\operatorname{End}}
\newcommand{\id}{\operatorname{id}}
\newcommand{\ind}{\operatorname{ind}}

\newcommand{\Irr}{\operatorname{Irr}}

\newcommand{\pprod}{\operatorname{prod}}
\newcommand{\reg}{\operatorname{reg}}
\newcommand{\Stab}{\operatorname{Stab}}
\newcommand{\SC}{\operatorname{sc}}
\newcommand{\GL}{\operatorname{GL}}
\newcommand{\SL}{\operatorname{SL}}

\newcommand{\CSp}{\operatorname{CSp}}
\newcommand{\Sp}{\operatorname{Sp}}
\newcommand{\SO}{\operatorname{SO}}
\newcommand{\Sym}{{\operatorname{S}}}

\newcommand{\cu}{{\operatorname{cusp}}}

\newcommand{\general}{} 
\newcommand{\Irrl}{\mathrm{Irr}_{\ell'}}
\def\norm#1#2{{\operatorname N}_{#1}(#2)}
\def\cent#1#2{{\operatorname C}_{#1}(#2)}

\newcommand{\w}{\widetilde}
\newcommand{\wh}{\widehat}
\newcommand{\wi}{\wh\iota}

\newcommand{\wG}{\ensuremath{{\w G}}}
\newcommand{\wN}{\ensuremath{{\w N}}}
\newcommand{\wT }{\ensuremath{\w T}}
\newcommand{\Z}{\operatorname Z}

\newcommand{\ovF}{\overline \FF}
\renewcommand{\o}{\overline }
\newcommand{\tw}[1]{{}^#1}
\newcommand{\bww}{{\w{\mathbf w_{0}}}}
\newcommand{\deq}{:=}
\newcommand{\GF}{{{\bG^F}}}
\newcommand{\wGF}{{{{\w\bG}^F}}}

\let\al=\alpha
\let\eps=\epsilon
\let\si=\sigma
\let\la=\lambda
\newcommand{\wla}{{\w\lambda}}
\newcommand\sirho{{^\sigma\!\rho}}
\newcommand\sila{{^\sigma\!\lambda}}
\newcommand\sieta{{^\sigma\!\eta}}

\let\co=\colon

\newtheorem{thm}{Theorem}[section]
\newtheorem{lem}[thm]{Lemma}
\newtheorem{cor}[thm]{Corollary}
\newtheorem{prop}[thm]{Proposition}

\newtheorem{thmA}{Theorem}

\theoremstyle{definition}

\newtheorem{defn}[thm]{Definition}
\newtheorem{notation}[thm]{Notation}

\theoremstyle{remark}
\newtheorem{rem}[thm]{Remark}

\numberwithin{equation}{section} 
\numberwithin{figure}{section} 

\makeatletter
\def\restr#1|#2{\left.#1\right\rceil_{#2}}
\def\spann<#1>{\left\langle#1\right\rangle}

\def\Spann<#1>{\Spann@h#1@}
\def\Spann@h#1|#2@{\left\langle\left.#1\vphantom{#2}\hskip.1em\right.\mid\relax #2 \right\rangle}
\def\Set#1{\Set@h#1@}
\def\Set@h#1|#2@{\left\{\left.#1\vphantom{#2}\hskip.1em\,\right.\mid\relax #2\right\}}
\makeatother

\begin{document}

\title{Characters of odd degree}

\date{\today}

\author{Gunter Malle}
\address{FB Mathematik, TU Kaiserslautern, Postfach 3049,
         67653 Kaisers\-lautern, Germany.}
\email{malle@mathematik.uni-kl.de}
\author{Britta Sp\"ath}
\address{FB Mathematik, TU Kaiserslautern, Postfach 3049,
         67653 Kaiserslautern, Germany}
\email{spaeth@mathematik.uni-kl.de}

\dedicatory{Dedicated to Gabriel Navarro, for his fundamental contributions}

\thanks{The authors gratefully acknowledge financial support by ERC
  Advanced Grant 291512.}

\keywords{McKay conjecture, characters of odd degree, equivariant Harish-Chandra theory}

\subjclass[2010]{Primary 20C15, 20C33; Secondary 20C08}

\begin{abstract}
We prove the McKay conjecture on characters of odd degree.
A major step in the proof is the verification of the inductive McKay condition
for groups of Lie type and primes $\ell$ such that a Sylow $\ell$-subgroup
or its maximal normal abelian subgroup is contained in a maximally split
torus by means of a new equivariant version of Harish-Chandra induction.
Specifics of
characters of odd degree, namely that they only lie in very particular
Harish-Chandra series then allow us to deduce from it the McKay conjecture
for the prime~$2$, hence for characters of odd degree.
\end{abstract}

\maketitle


\section{Introduction}

\noindent
In his 1972 note \cite{McK} dedicated to Richard Brauer on the occasion of
his 70th birthday John McKay put forward the following conjecture, based on
observations on the known character tables of finite simple groups
and of symmetric groups:
\begin{center}
 \emph{For a finite simple group $G$, $m_2(G)=m_2(\NNN_G(S_2))$, where\\
  $S_2$ is a Sylow 2-group of $G$.\qquad\qquad\qquad\qquad\qquad\qquad\quad\ }
\end{center}
Here, for a finite group $H$, $m_2(H)$ denotes the number of complex
irreducible characters of $H$ of odd degree. Soon after the appearance of
\cite{McK}, this observation was generalised to arbitrary finite groups and
primes. The \emph{McKay conjecture} thus claims that
for every finite group $G$ and every prime $\ell$ the number of ordinary
irreducible characters $\chi\in\Irr(G)$ with $\ell\nmid\chi(1)$ is locally
determined, namely $$|\Irrl(G)|=|\Irrl(\NNN_G(P))|,$$ where $P$ is a Sylow
$\ell$-subgroup of $G$ and $\NNN_G(P)$ denotes its normaliser in $G$.
The main result of our paper is the proof of that conjecture for \emph{all}
finite groups and the prime $2$.

\begin{thmA}   \label{thm:McKayp=2}
 Let $G$ be a finite group. Then the numbers of odd degree irreducible
 characters of $G$ and of the normaliser of a Sylow $2$-subgroup of $G$ agree.
\end{thmA}

McKay's conjecture had a decisive influence on the development of modern
representation theory of finite groups. Its prediction of how local structures
like the normaliser of a Sylow $\ell$-subgroup should influence the
representation theory of a group gave rise to a whole array of stronger and
more far reaching conjectures, like those of Alperin, of Brou\'e and of Dade.
Simultaneously, functors relating the
representation theory of certain families of finite (nearly simple) groups
with those of suitable subgroups were introduced. For example Harish-Chandra
induction and its generalisation by Deligne--Lusztig are key to the
parametrisation of characters of groups of Lie type.

It was Gabriel Navarro who, in his work and in his talks, insisted that the
McKay conjecture \lq lays at the heart of everything\rq. His insights led to
the proof of the fundamental result \cite{IMN} that the McKay conjecture holds
for all finite groups at a prime $\ell$, if every finite non-abelian
simple group satisfies a set of properties, the now so-called \emph{inductive
McKay condition}, for $\ell$. (A streamlined version of this reduction
was presented in \cite{Sp13} while a novel approach to groups with
self-normalising  Sylow $2$-subgroups was recently devised by Navarro and
Tiep \cite{NT15}.) This opens the possibility to solve the conjecture
through the classification of finite simple groups. Thanks to the work of
several authors this inductive condition has been shown for all but seven
infinite series of simple groups of Lie type $S$ at primes $\ell$ different
from the defining characteristic of $S$, see \cite{CS15,ManonLie,Sp12}. 

The second main result of our paper is meant to provide an important step
towards verifying the McKay conjecture in the case of odd primes, showing that
the inductive McKay condition holds for most simple groups of Lie type in the
maximally split case:

\begin{thmA}   \label{thm:d=1good}
 Let $\bG$ be a simple linear algebraic group of  simply connected type defined over
 $\FF_q$ with respect to the Frobenius endomorphism $F:\bG\rightarrow \bG$
 such that $S:=\bG^F/\Z(\bG^F)$ is simple. Assume that
 $S\notin\{\tD_{l,\SC}(q),\tE_{6,\SC}(q)\}$ for any prime power $q$.
 Then the inductive McKay condition from \cite[\S10]{IMN} holds for $S$ and
 all primes $\ell $ dividing $q-1$.
\end{thmA}

For many simple groups $S$ of Lie type and primes $\ell$ different from the
defining characteristic of $S$ the authors had constructed a bijection
satisfying some (but not all) of the required
properties from the inductive McKay condition. Moreover, in the cases
where the associated algebraic group has connected centre, Cabanes and the
second author \cite{CS13} could then verify the inductive McKay condition. It
thus remains to deal with simple groups of Lie type coming from
algebraic groups of simply connected type with disconnected centre.

One decisive  ingredient in our proof is a criterion for the inductive McKay
condition tailored to groups of Lie type, see \cite[Thm.~2.12]{Sp12}, which
we recall here in Theorem~\ref{thm:Sp12}. It had already been used for
groups of type $\tA_l$ as well as in the defining characteristic, see
\cite{CS15} and
\cite{Sp12}. The main assumption of that theorem on the universal covering
group $G$ of a simple group $S$ of Lie type and the prime $\ell$ consists of
three requirements:
\begin{itemize}
\item the global part concerns the stabilisers in the automorphism group and
 the extendibility of elements in $\Irrl(G)$, see assumption \ref{2_2glo};
\item for a suitably chosen subgroup $N$ that has properties similar to the
 normaliser of a Sylow $\ell$-subgroup of $G$, the elements of $\Irrl(N)$
 have only stabilisers of specific structures and have an analogous property
 with respect to extendibility, see assumption \ref{thm2_2loc};
\item there exists an equivariant global-local bijection between the
 relevant characters of certain groups containing $G$ and $N$, respectively,
 see assumption \ref{thm2_2bij}.
\end{itemize}

We successively establish those assumptions in the cases relevant to our
Theorems~\ref{thm:McKayp=2} and \ref{thm:d=1good}. In accordance with
\cite{MaH0}
we choose $N$ to be the normaliser of a suitable Sylow $d$-torus. Extending
earlier results of the second author we derive the required statement about
the stabilisers of local characters, see Section~\ref{sec:IrrlN}. We also
establish that a similar result holds in type $\tC_l$ for the characters of
the normaliser of a certain torus, see Section~\ref{type C}.
Afterwards we study the parametrisation of irreducible characters in terms of
Harish-Chandra induction, control how automorphisms act on these characters
and express this in terms of their labels, see Theorem \ref{thm:equiv_HC}.
The proof requires an equivariant version of Howlett--Lehrer theory describing
the decomposition of Harish-Chandra induced cuspidal characters and relies on
an extendibility result of Howlett--Lehrer and Lusztig. We then prove that
many characters of $G$ have stabilisers of the structure required in the
criterion.
\medskip

\noindent
{\bf Structure of the paper.}
After introducing some notation in Section~\ref{sec:Not}, we start by recalling
the parametrisation of characters of normalisers of Sylow $d$-tori for
$d\in \{1,2\}$ and describe how automorphisms act on the characters
and the associated labels in Section~\ref{sec:IrrlN}. This result has an
analogue for the characters of the normaliser of a certain torus in type
$\tC_l$, see Section~\ref{type C}.

In Section~\ref{sec:HC}, after recalling the basic results on the endomorphism
algebra of Harish-Chandra induced modules of $G$, we describe the action of
outer automorphisms $\sigma\in\Aut(G)$ on such modules.
This enables us in Theorem~\ref{thm:bij} to construct an equivariant
local-global bijection given by Harish-Chandra induction. The aforementioned
results on stabilisers of characters of the normaliser of a maximally split
torus leads to a description of the stabilisers of some characters of $G$ in a
similar way, see Corollary~\ref{cor:7_3}.

The remaining part of the paper is devoted to the completion of the proof of
our main Theorems~\ref{thm:McKayp=2} and~\ref{thm:d=1good}. First we show that
all necessary assumptions of Theorem~\ref{thm:Sp12} are satisfied for proving
Theorem~\ref{thm:d=1good} and clarify which additional properties need to
be proved for obtaining an even more general statement. Then, after the
classification of odd degree characters of quasi-simple groups of Lie type in
Theorem~\ref{thm:odd degree} which may be of independent interest, we complete
the proof of Theorem~\ref{thm:McKayp=2}.
\medskip

\noindent{\bf Acknowledgement.}
The second author thanks Michel Enguehard for comments on an early draft of
results now contained in Section~\ref{sec:HC}.

\section{Background}   \label{sec:Not}
\noindent
In this section we recall the criterion from \cite{Sp12} for the inductive
McKay condition that is the main tool in the proof of our main result.
Afterwards we introduce the groups of Lie type that play a central role in
the paper and describe their automorphisms.

\subsection{A criterion for the  inductive McKay condition} 
We first introduce some notation.
If a group $A$ acts on a finite set $X$ we denote by $A_{x}$ the stabiliser of
$x\in X$ in $A$, analogously we denote by $A_{X'}$ the setwise stabiliser of
$X'\subseteq X$. For an element $a\in A$ we denote by $o(a)$ the order of $a$.
If $A$ acts on a group $G$ by automorphisms, there is a natural action of $A$
on $\Irr(G)$ given by
\[ {}^{a^{-1}}\chi (g)=\chi^a(g)=\chi(g^{a^{-1}})\quad
  \text{ for every } g \in G,\,\, a\in A \text{ and } \chi\in\Irr(G).\]
For $P\leq G$ and $\chi\in \Irr(H)$ for some $A_P$-stable subgroup $H\leq G$,
we denote by $A_{P,\chi}$ the stabiliser of $\chi$ in $A_P$.

We denote the restriction of $\chi\in\Irr(G)$ to a subgroup $H\leq G$ by
$\restr \chi|H$, while $\chi^G$ denotes the character induced from
$\psi\in\Irr(H)$ to $G$. For $N\lhd G$ and $\chi\in \Irr(G)$ we denote by
$\Irr(N\mid \chi)$ the set of
irreducible constituents of the restricted character $\restr\chi|N$, and for
$\psi \in \Irr(N)$, the set of irreducible constituents of the induced
character $\psi^G$ is denoted by $\Irr(G\mid \psi)$. For a subset
$\cN\subseteq \Irr(N)$ we define
\[ \Irr(G\mid \cN)\deq\bigcup_{\chi\in\cN}\Irr(G\mid \chi).\]
Additionally, for $N\lhd G$ we sometimes identify the
characters of $G/N$ with the characters of $G$ whose kernel contains $N$.
For a prime $\ell$ we let
$\Irr_{\ell'}(G)\deq\{\chi\in\Irr (G)\mid \ell\nmid\chi(1)\}$.

The following criterion was proved in Sp\"ath {\cite[Thm.~2.12]{Sp12}}:

\begin{thm}   \label{thm:Sp12}
 Let $S$ be a finite non-abelian simple group and $\ell$ a prime dividing
 $|S|$. Let $G$ be the universal covering group of $S$ and $Q$ a Sylow
 $\ell$-subgroup of $G$. Assume there exist groups $A$, $\w G\leq A$, $D\leq A$
 and $N\lneq G$, such that with $\wN\deq N\NNN_{\wG}(Q)$ the
 following conditions hold:
 \begin{enumerate}[label=\rm(\roman*),ref=\thethm(\roman{enumi})]
  \item \label{thm2_2gen}
  \begin{enumerate}[label=\rm(\arabic*),ref=\thethm(\roman{enumi}.\arabic{enumii})]
   \item $G\lhd A$, $G\le\wG$ and $A=\wG \rtimes D$,
   \item $\wG/G$ is abelian,
   \item $\Cent_{\wG\rtimes D}(G)= \Z(\wG)$ and $A/\Z(\wG)\cong\Aut(G)$
    by the natural map,
   \item $N$ is $\Aut(G)_Q$-stable,
   \item $\NNN_G(Q)\leq N$,
   \item  \label{hauptprop_maxext_glob}
	every $\chi\in\Irrl(G)$ extends to its stabiliser $\w G_\chi$,
   \item  \label{hauptprop_maxext_loc}
	every $\psi\in \Irrl(N)$ extends to its stabiliser $\wN_\psi$.
  \end{enumerate}
  \item  \label{2_2glo}
   Let $\cG\deq \Irr\big (\w G\mid \Irrl(G)\big)$. For every $\chi\in\cG$
   there exists some $\chi_0\in \Irr (G\mid \chi)$ such that
  \begin{enumerate}[label=\rm(\arabic*),ref=\thethm(\roman{enumi}.\arabic{enumii})]
   \item  \label{2_2glostar}
    $(\w G\rtimes D)_{\chi_0}= \w G_{\chi_0}\rtimes D_{\chi_0}$ and
   \item  \label{2_2gloext}
    $\chi_0$ extends to $(G \rtimes D)_{\chi_0}$.
  \end{enumerate}
  \item  \label{thm2_2loc}
   Let $\cN\deq \Irr\big (\w N\mid \Irrl(N)\big )$. For every $\psi\in \cN$
   there exists some $\psi_0\in \Irr(N\mid \psi)$ such that
   $O\deq G(\w G\rtimes D)_{N,\psi_0}$ satisfies
   \begin{enumerate}[label=\rm(\arabic*),ref=\thethm(\roman{enumi}.\arabic{enumii})]
    \item  \label{2_2locstar}
     $O=(\w G\cap O) \rtimes (D\cap O)$ and
   \item  \label{2_2loc-ext}
    $\psi_0$ extends to $(G\rtimes D)_{N,\psi_0}$.
   \end{enumerate}
   \item\label{thm2_2bij} There exists a $(\wG\rtimes D)_Q$-equivariant
    bijection
	$\w \Omega: \cG \longrightarrow \cN$
	with
    \begin{enumerate}[label=\rm(\arabic*),ref=(\roman{enumi}.\arabic{enumii})]
     \item ${\w \Omega}(\cG\cap\Irr(\w G\mid \nu))=\cN\cap\Irr(\w N\mid \nu)$
       for every $\nu \in \Irr(\Z(\w G))$,
     \item  \label{Omega_u_epsilon_equiv}
      ${\w \Omega}(\chi\delta)= {\w\Omega}(\chi)\restr\delta|{\w N}$ for every
      $\chi\in \cG$ and every $\delta\in\Irr(\w G|1_G)$.
   \end{enumerate}
 \end{enumerate}
 \smallskip
 Then the inductive McKay condition from \cite[\S 10]{IMN} holds for $S$ and
 $\ell$.
\end{thm}

\subsection{Simple groups of Lie type}\label{ssec2:B}
We now introduce the most relevant groups and automorphisms. For the later
detailed calculations it is relevant to fix them in a rather precise way.
Let $\bG$ be a simple linear algebraic group of simply connected type over an
algebraic closure of $\FF_q$. Let $\bB$ be a Borel subgroup of $\bG$ with
maximal torus $\bT$. Let $\Phi,\Phi^+$ and $\Delta$ denote the set of roots,
positive roots and simple roots of $\bG$ that are determined by $\bT$ and
$\bB$. Let $\bN:=\NNN_\bG(\bT)$. We denote by $W$ the Weyl group of $\bG$ and
by $\pi:\norm\bG\bT\rightarrow W$ the defining epimorphism. For calculations
with elements of $\bG$ we use the Chevalley generators subject to
the Steinberg relations as in \cite[Thm.~1.12.1]{GLS3}, i.e., the elements
$x_\al(t)$, $n_\al(t)$ and $h_\al(t)$ ($t\in \ovF_q$ and $\al\in \Phi$)
defined as there.

In the following we describe automorphisms of $\bG$.
Let $p$ be the prime with $p\mid q$ and $F_0: \bG\rightarrow \bG$ the
\emph{field endomorphism} of $\bG$ given by
\[F_0(x_\al(t))= x_\al(t^p) \quad\text{ for every } t \in \ovF_q
   \text{ and } \al \in \Phi.\]
Any length-preserving automorphism $\tau$ of the Dynkin diagram associated to
$\Delta$ and hence automorphism of $\Phi$ determines a \emph{graph
automorphism} $\gamma$ of $\bG$ given by
\[ \gamma(x_\al(t))=x_{\tau(\al)}(t) \quad\text{ for every } t \in \ovF_q
   \text{ and } \al \in \pm \Delta.\]
Note that any such $\gamma$ commutes with $F_0$.

For the construction of diagonal automorphisms of the associated finite groups
of Lie type we introduce further groups: Let $r$ be the rank of $\Z(\bG)$ (as
abelian group) and $\bZ\cong (\ovF_q^\times)^r$ a torus of that rank with an
embedding of $\Z(\bG)$. We set
\[ \w\bG:= \bG \times_{\Z(\bG)} \bZ,\]
the central product of $\bG$ with $\bZ$ over $\Z(\bG)$. Then $\w\bG$ is
a connected reductive group with connected centre and the natural map
$\bG\rightarrow \w \bG$ is a regular embedding, see \cite[15.1]{CE04}.
Note that $\w\bB:=\bB\bZ$ is a Borel subgroup of $\w\bG$ and $\w\bT:=\bT\bZ$
is a maximal torus therein. Furthermore let
$\w \bN:=\NNN_{\w\bG}(\w\bT)=\bN \bZ$.

As $F_0$ acts on $Z(\bG)$ via $x\mapsto x^p$ for every $x\in\Z(\bG)$
we can extend it to a Frobenius endomorphism $F_0: \w \bG \rightarrow\w \bG$
via
\[ F_0(g,x):= (F_0(g), x^p) \quad\text{ for every }g\in \bG
   \text{ and }x \in \bZ.\]
Now assume that $\gamma$ is a graph automorphism of $\bG$. If $\gamma$ acts
trivially on $\Z(\bG)$ then it extends to an automorphism of $\w\bG$ which we
also denote by $\gamma$, via
\[ \gamma(g,x):= (\gamma(g), x) \quad\text{ for every }g\in \bG
   \text{ and }x \in \bZ.\]
If $\gamma$ acts on $\Z(\bG)$ by inversion then it can be extended via
\[ \gamma(g,x):= (\gamma(g), x^{-1}) \quad\text{ for every }g\in \bG
   \text{ and }x \in \bZ.\]
A similar extension of $\gamma$ is possible in the remaining cases.
In any case $F_0$ and $\gamma$ stabilise $\w\bB$ and $\w\bT$.

Now consider a Steinberg endomorphism $F:=F_0^m\gamma$, with $\gamma$ a
(possibly trivial) graph automorphism of $\bG$. Then $F$ defines an
$\FF_q$-structure on $\w\bG$, where $q=p^m$, and $\bB,\bT,\w\bB,\w\bT$ are
$F$-stable, so in particular $\bT,\w\bT$ are maximally split tori in $\bG$,
$\w\bG$ respectively. We let $G:=\bG^F$. By construction the order of
$F_0$ as automorphism of $\wG:=\w\bG^F$ coincides with the one
of $F_0$ as automorphism of $G$. The analogous statement also holds for any
graph automorphism $\gamma$ and the automorphisms of $\w G$ associated with it.

Let $D$ be the subgroup of $\Aut(G)$ generated by $F_0$ and the graph
automorphisms commuting with $F$. Then $\w G\rtimes D$ is well-defined and
induces all automorphisms of $G$, see \cite[Thm.~2.5.1]{GLS3}. Moreover $D$
acts naturally on the set of $F$-stable subgroups of $\bG$.

\subsection{An embedding of the group $\tD_{l,sc}(q)$ into $\tB_{l,sc}(q)$}
\label{embed_D_into_B}
We recall an embedding of $\tD_{l,sc}(q)$ into $\tB_{l,sc}(q)$ given explicitly
in \cite[10.1]{Spaeth2} in terms of the aforementioned Chevalley generators.
Let $\o \Phi$ be a root system of type $\tB_l$ with base
$\Delta=\{\o \alpha_1, \alpha_2,\ldots, \al_l \}$, where
$\o\al_1=e_1$ and $\al_i=e_i-e_{i-1}$ ($i\geq 2$) as in \cite[Rem.~1.8.8]{GLS3}.
Let $\o \bG$ be the associated simple algebraic group of simply connected type over
$\o\FF_q$. In analogy to our previous terminology we denote its Chevalley
generators by $\o x_\al(t_1)$, $\o n_\al(t_2)$ and $\o h_\al(t_2)$ with
($\al\in\o \Phi$, $t_1\in \o\FF_q$ and $t_2\in\o \FF_q^\times$).

Let $\Phi\subseteq \o \Phi$ be the root system consisting of all long roots
of $\o\Phi$. Then the group $\spann<x_\al(t)\mid \al \in \Phi,\,t\in\o\FF_q>$
is a simply connected simple
group over $\o\FF_q$ with the root system $\Phi$ of type $\tD_l$.

Whenever $\Phi$ is of type $\tD_l$, we identify $\bG$ with
$\spann<\o x_\al(t)\mid \al \in \Phi,t\in\o\FF_q>$ via $\iota_{\tD}:\bG
\rightarrow \o \bG$, $x_\al(t)\mapsto \o x_\al(t)$,
and choose the notation of elements in $\bG$ such that this defines a
monomorphism.
Let $\zeta\in \o\FF_q$ be a primitive $(2,q-1)^2$th root of unity.
The graph automorphism of $\bG$ of order~2 coincides with the map $x\mapsto
x^{\o n_{e_1}(1) \prod_{i=2}^l \o h_{e_i}(\zeta)}$,
see \cite[Lemma~11.2]{Spaeth2}, which because of
$\Z(\bG)=\spann<\o h_{e_1}(-1), \prod_{i=1}^l \o h_{e_i}(\zeta)>$
(by \cite[Tab.~1.12.6 and Thm~1.12.1(e)]{GLS3})
coincides with $x\mapsto x^{\o n_{e_1}(1) \o h_{e_1}(\zeta)}$.

\section{Parametrisation of some local characters}   \label{sec:IrrlN}
\noindent
In this section we prove a result on stabilisers of characters that leads to
the verification of condition~\ref{thm2_2loc} in the cases considered in this
paper. These results enable us later in Theorem~\ref{thm:Bij_wG} to construct a
bijection $\w\Omega:\cG\rightarrow\cN$ as required in Theorem~\ref{thm2_2bij}.

The aim of this section is the proof of the following statement that concerns
normalisers of Sylow $d$-tori, sometimes also called Sylow $d$-normalisers.
Sylow $d$-tori were introduced in \cite{BM92} under the name of Sylow
$\Phi_d$-tori (with $\Phi_d$ denoting the $d$-th cyclotomic polynomial), and
play an important role in the study of height $0$ characters, see \cite{MaH0}.

\begin{thm}   \label{thm:IrrN_autom}
 Let $d\in\{1,2\}$, $\bS_0$ be a Sylow $d$-torus of $(\bG,F)$,
 $N_0\deq \NNN_\bG(\bS_0)^F$, $\w N_0\deq \NNN_{\w \bG}(\bS_0)^F$ and
 $\psi\in\Irr(\w N_0)$.
 There exists some $\psi_0\in \Irr(N_0\mid \psi)$ such
 that
 \begin{enumerate}[label=\rm(\arabic*),ref=(\arabic{enumi})]
  \item $O_0= (\wGF\cap O_0) \rtimes (D\cap O_0)$ for
   $O_0\deq \GF(\wGF\rtimes D)_{\bS_0,\psi_0}$; and
  \item $\psi_0$ extends to $(\GF\rtimes D)_{\bS_0,\psi_0}$.
 \end{enumerate}
\end{thm}

This statement  is related to Theorem 5.1 of \cite{CS15}, where the same
assertion was proved for all positive integers $d$ in the case that the root
system of $\bG$ is of type $\tA_l$. Accordingly we may and will assume in the
following that $\Phi$ is not of type $\tA_l$.

We verify the statement in five steps mimicking the strategy applied in
\cite[Sec.~5]{CS15}. First, in \ref{sec:3:Transfer} we replace $\GF$ by an
isomorphic group, then for subgroups of this group we construct in
\ref{sec:3:extmap} an extension map that is compatible with certain
automorphisms of $\GF$, which gives in \ref{sec:3_param} a parametrisation of
$\Irr(N_0)$. In the end, the condition~\ref{2_2locstar} on the structure of
stabilisers is deduced from properties of characters of relative inertia groups.

By what we said before we may and will also assume throughout this section that
$D$ is non-trivial and that $\w\bG$ induces non-inner automorphisms on $\bG$.
Accordingly the root system $\Phi$ of $\bG$ is of type $\tB_l$, $\tC_l$,
$\tD_l$, $\tE_6$ or $\tE_7$ and $\Z(\bG^F)\neq 1$, hence in particular
$\GF\neq \tw3\tD_{4,\SC}(q)$.

\subsection{Transfer to twisted groups}   \label{sec:3:Transfer}

Recall the notations  from Section~\ref{sec:Not}.
We set $V:=\langle n_\al(\pm 1)\mid \alpha \in \Phi\rangle\le \NNN_\bG(\bT)$,
and $H:=V\cap\bT$. We define $v\in \bG$ as
\begin{align*}
  v&:=\begin{cases} \id_{\bG}& \text{if }d=1,\\ \bww& \text{if }d=2,\end{cases}
\end{align*}
where $\bww$ is the canonical representative in $V$ of the longest element
of $W$ defined as in \cite[Def.~3.2]{Spaeth2}.

\begin{lem}   \label{lem:3_3}
 The torus $\bT$ contains a Sylow $d$-torus $\bS$ of $(\bG, vF)$. Moreover
 $\bT=\Cent_\bG(\bS)$ and $N= T V_1$, where
 $N:=\norm \bG \bS^{vF}$, $T\deq \bT^{vF}$ and $V_1:=V^{vF}$.
\end{lem}

\begin{proof}
Let $\phi$ denote the automorphism induced by $F$ on $W$. Comparing with the
tables in \cite[Sect.~5 and 6]{Springer} one sees that $\pi(v)\phi$ is a
$d$-regular element of $W\phi$ in the sense of Springer, see
\cite[Sect.~4 and 6]{Springer}. Hence the centraliser of any Sylow $d$-torus
in $\bG$ is a torus.

According to \cite[Rem.~3.3 and Lemma~3.4]{Spaeth2} there exists some
Sylow $d$-torus $\bS\leq \bT$ of $(\bG, vF)$. If $\Phi$ is of classical type
and $F=F_0^m$ then $TV_1=N$ by \cite[Rem.~3.3(c)]{Spaeth2}. For exceptional
types this was proven in \cite[Prop.~6.3 and~6.4]{Sp09}.

It remains to consider the case where $\GF=\tw2\tD_{l,\SC}(q)$. Here for
$d=1$ one uses \cite[Lemma~11.2]{Spaeth2} and computes that $H^F$ is an
elementary abelian group of rank $l-1$ and that $\pi(V_1)$ is isomorphic to a
Coxeter group of type $\tB_{l-1}$ and hence to $\cent W \phi$. One can see
analogously for $d=2$ and hence $v=\bww$ that $H^{v F}$ is an elementary
abelian group of rank $l-1$, and $\pi(V ^{v F}) = \cent W {\pi(v) \phi}$ if
$w_0\in\Z(W)$ and hence $v\in\Z(V)$. If $w_0\notin \Z(W)$ computations in the
braid group show that $H^{vF}=H$ and $V^{vF}=V$.
\end{proof}

\begin{notation}   \label{not:3.4}
Let $e:=o(v)$, the order of $v$. In the following we denote by $\rC_i$ the
cyclic group of order $i$. Let $E_1$ be the subgroup of $\Aut(\bG)$ generated
by graph automorphisms. Let $E:= \rC_{2em} \times E_1$ act on
${\w \bG}^{F_0^{2em}}$ such that the first summand $\rC_{2em}$ of $E$ acts
by $\spann<F_0>$ and the second by the group generated by graph automorphisms.
Note that this action is faithful. Let $\wh F_0, \wh \gamma, \wh F \in E$ be
the elements that act on ${\w \bG}^{F^{2em}_0}$ by $F_0$, $\gamma$ and $F$,
respectively.

Note that $E$ stabilises $N$, $T$, $V$, $v$ and hence $H$, $V_1$ and $H^{vF}$.
\end{notation}

\begin{prop}   \label{prop:5_6}
 Let $\bS$ and $N$ be as in Lemma~\ref{lem:3_3}, and
 $\wN\deq \norm {\w\bG} \bS^{vF}$. Suppose that for
 every $\chi\in\Irr(\w N)$ there exists some $\chi_0\in\Irr(N\mid \chi)$ such
 that
 \begin{enumerate}[label=\rm(\arabic*),ref=(\arabic{enumi})]
  \item $(\w N \rtimes E)_{\chi_0} =\w N_{\chi_0} \rtimes E_{\chi_0}$;
   and
  \item $\chi_0$ has an extension $\w\chi_0\in\Irr(N\rtimes E_{\chi_0})$
   with $v\wh F\in\ker(\w \chi_0)$.
 \end{enumerate}
 Then the conclusion of Theorem~\ref{thm:IrrN_autom} holds for $(\bG,F)$ and
 $d$.
\end{prop}

\begin{proof}
The statement is an analogue of \cite[Prop.~5.3]{CS15}. The proof given there
is independent of the underlying type, and is based on the application of
Lang's theorem using that $v$ is $D$- and hence $E$-invariant. It relies on
the fact that conjugation by a suitable element of $\bG$ gives an isomorphism
$\iota: \bG\rightarrow\bG$ with $\iota(\bG^F)=\bG^{vF}$. Via $\iota$ the
automorphisms of $\bG^F$ induced by $\w\bG^F \rtimes D$ coincide with the ones
of $\bG^{vF}$ induced by $\w\bG^{vF}E/\langle v\wh F\rangle$, and
$\w\bG^F\rtimes D \cong \w\bG^{vF}E/\langle v\wh F\rangle$.
\end{proof}

\subsection{Extension maps with respect to $H_1\lhd V_1$}  \label{sec:3:extmap}
In order to verify the assumptions of Proposition~\ref{prop:5_6} on the
characters of $N$ we label them via some so-called extension map.

\begin{defn}[Definition 5.7 of \cite{CS15}]   \label{def3_6}
 Let $Y\lhd X$ and $\cY\subseteq \Irr(Y)$. We say that
 \emph{maximal extendibility holds for $\cY$ with respect to $Y\lhd X$}
 if every $\chi \in \cY$ extends (as irreducible character) to $X_\chi$.
 Then, an \emph{extension map for $\cY$ with respect to $Y\lhd X$} is a map
 \[\Lambda: \cY \rightarrow \bigcup_{Y\leq I\leq X} \Irr(I),\]
 such that for every $\chi\in \cY$ the character
 $\Lambda(\chi)\in \Irr(X_\chi)$ is an extension of $\chi$.
 If $\cY=\Irr(Y)$ we also say that there exists \emph{an \general extension map
 with respect to $Y\lhd X$}.
\end{defn}

The following is easily verified:

\begin{lem}   \label{lem:notsec3}
 Let $X$ be a finite group, $Y\lhd X$ and $\cY\subseteq \Irr(Y)$ an $X$-stable
 subset. Assume there exists an extension map for $\cY$ with respect
 to $Y\lhd X$. Then there exists an $X$-equivariant extension map for $\cY$
 with respect to $Y\lhd X$.
\end{lem}

In order to prove Theorem~\ref{thm:IrrN_autom} in the form suggested by
Proposition~\ref{prop:5_6} our next goal is to establish the following
intermediary step. Recall $V_1=V^{vF}$ and set $H_1:=H^{vF}$.

\begin{thm}   \label{thm:very_good_twist}
 There exists a $V_1E$-equivariant \general extension map 
 with respect to $H_1\lhd V_1$.
\end{thm}

The proof will be given in several steps. We first consider the case when $F$
is untwisted and $d=1$.

\begin{prop}   \label{prop:very_good_twist_1}
 For $\Phi$ not of type $\tD_l$ there exists an \general extension map with
 respect to $H\lhd V$.
\end{prop}

\begin{proof}
According to \cite[Prop.~5.1]{Sp09} we can
assume that $\Phi$ is of type $\tB_l$ or $\tC_l$.
Assume that $q=3$. Then $V=\bN^F$ and $H=\bT^F$. Maximal extendibility holds
with respect to $H=\bT^F\lhd \bN^F$ according to 
\cite[Cor.~6.11]{HL} or \cite[Thm.~1.1]{Spaeth2}.

By assumption $q$ is odd. Then the isomorphism types of $H$ and $V$ are
independent of $q$ since $V$ and $H$ can be described as finitely presented
groups whose relations are independent of $q$, see \cite{Tits} and
\cite[Lemma~2.3.1(b)]{Sp_Diss}. Hence the considerations for $q=3$ already
imply the statement.
\end{proof}

\begin{prop}   \label{prop:3_10}
 For $\Phi$ not of type $\tD_l$ there exists a $VE_1$-equivariant \general
 extension map with respect to $H\lhd V$.
\end{prop}

\begin{proof}
If $\Phi$ has no graph automorphism Proposition \ref{prop:very_good_twist_1}
together with Lemma~\ref{lem:notsec3} proves that a $V$-equivariant \general
extension map with respect to $H\lhd V$ exists.

If $\Phi$ is of type $\tE_6$ the generator $\wh \gamma$ of $E_1$ corresponds
to an automorphism of the associated braid group $\tB$, that acts by permuting
the generators. The epimorphism $\tau:\tB\rightarrow V$ is
$\wh\gamma$-equivariant. Let
$\rb:W\rightarrow \tB$ be the map from \cite[4.1.1]{GP} and $w_0$ the longest
element in $W$. Conjugation with $\rb(w_0)={\mathrm w}_0$ acts on $\bB$ like
$\wh \gamma$ by \cite[Lemma~4.1.9]{GP}, analogously conjugating by $\bww$, which
is the image of $\rb(w_0)$ under the natural epimorphism from $\bB$ to $V$,
acts on $V$ like $\wh \gamma$. Hence the automorphism induced by $\wh \gamma$ on
$V$ is an inner automorphism and hence any $V$-equivariant extension map
is also $VE_1$-equivariant.
\end{proof}

\begin{prop}   \label{prop:3_Dl}
 If $\Phi$ is of type $\tD_l$ there exists a $VE_1$-equivariant \general
 extension map with respect to $H\lhd V$.
\end{prop}

\begin{proof}
First let us consider the case where $\Phi$ is of type $\tD_l\neq \tD_4$.
Let $\iota_{\tD}:\bG\rightarrow \o\bG$ be the embedding from
\ref{embed_D_into_B}. Then
$\iota_{\tD}(V)\leq \o V \deq\spann<\o n_{\al}(\pm 1)\mid \al \in\o \Phi>$ and
$\iota_{\tD}(H)=\o H\deq \spann< \o h_{\al}(\pm 1)\mid \al \in \o \Phi>$.
Note that $\o H=H$ and hence $V_\la\leq \o V_\la$ for every  $\la\in \Irr(H)=\Irr(\o H)$.
Let $\Lambda_{\tB}$ be the $\o V$-equivariant \general extension map
with respect to $\o H\lhd \o V$ from Proposition \ref{prop:very_good_twist_1}.

As explained in \ref{embed_D_into_B}, $\gamma(x)=x^{\o n_{e_1}(1)
\o h_{e_1}(\zeta)}$ for every $x\in \bG$, where $\zeta$ is some primitive
$8$th root of unity. (Note that because of our initial reductions we can
assume that $2\nmid q$.) Let $\zeta'\in\o\FF_q$ be a primitive $8$th root of
unity and $t:=\prod_{i=1}^{l} h_{e_i}(\zeta')$. For $n\in V$ we have
$$ [t,n]= \prod_{j\in J} h_{e_j}(\zeta'^{2})$$
for a set $J\subseteq \{1,\ldots, l\}$ with $2\mid |J|$.
This proves $[t,V]\subseteq H$ and hence $V^t=V$.

Hence there is a well-defined \general extension map $\Lambda_0$ given by
$$\Lambda_0(\la)=(\restr\Lambda_{\tB}(\la)|{V_\la})^t\qquad \text{ for all }
  \la \in \Irr(H).$$
Since $\Lambda_{\tB}$ is $\o V$-equivariant, $\Lambda_0$ is
$\o V^t$-equivariant. The element
$\o n_{\al_1}(1)^t=\o n_{\al_1}(1) h_{\al_1}(\zeta'^2)$ and $\gamma$ induce the
same automorphism on $V$, according to \ref{embed_D_into_B}. Hence $\Lambda_0$
is $V\spann<\gamma>$-equivariant.

Now assume that $\Phi$ is  of type $\tD_4$. According to the above
considerations there exists some $V\spann<\wh \gamma_2>$-equivariant extension
map with respect to $H\lhd V$ for some $\wh \gamma_2\in E_1$ of order $2$.
For the proof it is sufficient to show maximal extendibility for some
$VE_1$-transversal $\TT\subset \Irr(H)$ with respect to $H\lhd VE_1$.
We may choose $\TT$ such that for each $\la \in \TT$ some Sylow $2$-subgroup of
$(VE_1)_\la$ is contained in $(V\spann<\wh \gamma_2>)_\la$.

According to \cite[Thm.~6.26]{Isa}, every $\la\in \TT$ extends to $(VE_1)_\la$
if $\la$ extends to a Sylow $2$-subgroup of $(VE_1)_\la$. By the choice of
$\TT$ we have $(VE_1)_\la=(V\spann<\wh \gamma_2>)_\la$ for every $\la\in \TT$.
By the above $\la$ has a $(V\spann<\wh\gamma_2>)$-invariant extension to
$V_\la$. Since $(V\spann<\wh\gamma_2>)_\la/V_\la$ is cyclic, $\la$ extends
to $(V\spann<\gamma_2>)_\la$ by \cite[Cor.~11.22]{Isa}. This proves the claim.
\end{proof}

In the next step we construct extension maps in the case where the Frobenius
endomorphism is twisted. Recall $V_1=V^{vF}$ and $H_1=H^{vF}$.

\begin{lem}   \label{lem:A3}
 Let $F_0$, $m$ and $\gamma$ be defined as in  \ref{ssec2:B}.
 Assume that $\Phi$ is of type $\tD_l$, $v=\id_\bG$ and $F=\gamma F_0^m$.
 Then there exists a $V_1 E_1$-equivariant \general extension
 map with respect to $H_1\lhd V_1$.
\end{lem}

\begin{proof}
By the proof of Lemma \ref{lem:3_3}, $V_1/H_1$ is isomorphic to $\rC_W(\gamma)$.

Let $\Delta=\{\al_1,\ldots ,\al_l\}$ be a base of $\Phi$. For a positive
integer $i$ and elements $x,y\in V$ let $\pprod(x,y, i)$ be defined by
\[ \pprod(x,y, i)=\underbrace{x\cdot y \cdot x \cdot y \cdots}_i. \]
Following \cite{Tits} the group $V$ coincides with the extended Weyl group
of $\bG$ that is the finitely presented group generated by $n_i=n_{\al_i}(1)$
and $h_i=n_i^2$, subject to the relations
\begin{align*}
  h_i h_j&=h_j h_i,& h_i^{2}&=1,\\
  \pprod(n_i,n_j, m_{ij})&=\pprod(n_j,n_i, m_{ij}),&
  h_i^{n_j}&=h_j^{A_{i,j}}h_i \text{ for all } 1\leq i,j\leq l,
\end{align*}
where $m_{ij}$ is the order of $s_{\al_i} s_{\al_j}$ in $W$ and $(A_{i,j})$
is the associated Cartan matrix, see \cite[Lemma~2.3.1(b)]{Sp_Diss} for
more details.

Assume that $\Delta$ is chosen such that the graph automorphism
$\gamma$ of order $2$ permutes $\al_1$ and $\al_2$. Using straight-forward
calculations one sees that the elements $n'_2:=n_{\al_1}(-1) n_{\al_2}(-1)$ and
$n'_i:=n_{\al_i}(-1)$ for $i>2$ satisfy the defining relations of an extended
Weyl group of type $\tB_{l-1}$. As the orders of the groups coincide, they
are isomorphic. Together with Proposition~\ref{prop:very_good_twist_1} this
implies the existence of the required extension map.

Note that according to Lemma \ref{lem:notsec3} the extension map can be chosen
to be $V_1$-equivariant. Since by definition $\gamma$ acts trivially on $V_1$
the extension map is also $V_1E_1$-equivariant.
\end{proof}

\begin{proof}[Proof of Theorem~\ref{thm:very_good_twist}]
The statement follows from the existence of a $V_1E_1$-equivariant \general
extension map since $\wh F_0$ acts trivially on $V$.

If $\Phi$ is of type $\tE_6$ the claim is implied by \cite[Lemma~8.2]{Sp09}.
In the remaining cases Propositions~\ref{prop:3_10} and
\ref{prop:3_Dl}, and Lemma~\ref{lem:A3} imply the statement
if $d=1$.

If $d=2$ and $\Phi$ is of type $\tB_l$, $\tC_l$ or $\tE_7$ the proof of
\cite[Lemma~6.1]{Sp09} shows that $v\in \Z(V)$. Hence $H=H^{vF}=H_1$ and
$V=V^{vF}=V_1$. Then Proposition~\ref{prop:very_good_twist_1} yields the claim.

The only remaining case is when $\Phi$ is of type $\tD_l$ and $d=2$.
Computations in $V$ show that either $V_1=V$ or $V_1=\rC_V(\gamma)$ and then
the statement about the maximal extendibility follows from the observations
made for $d=1$ in Proposition~\ref{prop:3_Dl} and Lemma~\ref{lem:A3}.
Hence there exists a $V_1E_1$-equivariant \general extension
map with respect to $H_1\lhd V_1$ in all cases.
\end{proof}

We next state a lemma helping to construct extensions with specific properties.

\begin{lem}   \label{lem:3_13}
 Let $\la\in \Irr(H_1)$ and $\w\la\in\Irr(V_{1,\la})$ a $(V_1E)_\la$-invariant
 extension of $\la$. Then $\w\la$ has an extension
 $\wh \la\in \Irr((V_1 E)_{\la})$ with $\wh\la(v\wh F)=1$.
\end{lem}

\begin{proof}
Recall $E_1=\spann<\wh\gamma>\leq E$ when $\GF\neq \tD_{4,\SC}(q)$ and
$E_1=\spann<\wh\gamma_2,\wh\gamma_3>$ otherwise, with $\wh\gamma_i$ of
order~$i$. Note that $\wh F_0$ is central
in $V_1\rtimes E$, i.e., $V_1E=(V_1\rtimes E_1)\times \langle\wh F_0\rangle$.
Since all Sylow subgroups of $E_1$ are cyclic, $\w\la$ extends to a character
$\psi$ of $(V_1E_1)_{\wla}$. Note that $(V_1E_1)_{\w\la}=(V_1E_1)_{\la}$.

Recall that $\GF\neq \tw 3 \tD_{4,\SC}(q)$ and $v\wh F=v \kappa \wh F_0^m$ for
some $\kappa \in E_1$. We have $o(\wh F_0^m)=2o(v)$ by the definition of $E$
and $o(v\kappa)\mid (2o(v))$ since $v$ and $\kappa$ commute. Accordingly there
exists some character $\epsilon\in\Irr(\langle\wh F_0\rangle)$ with
$\psi(1)\epsilon(\wh F_0^m)=\psi(v\kappa)^{-1}$. The character $\wh\la=
\psi\times \epsilon$ is an extension of $\w\la$ with the required properties.
\end{proof}

\subsection{Parametrisation of $\Irr(N)$}   \label{sec:3_param}
For the later understanding of the characters of $\Irr(N)$ we construct
an extension map with respect to $T\lhd N$. Recall $N\deq \NNN_\bG(\bS)^{vF}$
and $T\deq \Cent_{\bG}(\bS)^{vF}=\bT^{vF}$.

\begin{cor}   \label{cor:ker_delta_sc}
 There exists an \general extension map $\Lambda$
 with respect to $T\lhd N$ such that
  \begin{enumerate}[label=\rm(\arabic*),ref=(\arabic{enumi})]
  \item $\Lambda$ is $N\rtimes E$-equivariant; and
  \item for every $\la\in \Irr(T)$, there exists some linear $\w\la \in
   \Irr\left ((N\rtimes E)_{\la}\mid\Lambda(\la)\right )$ with
   $\w\la(v\wh F)=1$.
  \end{enumerate}
\end{cor}

Note that the existence of $\Lambda$ (without the properties required here)
is known from \cite[Cor.~6.11]{HL} for $d=1$, and from \cite{Sp09} and
\cite{Sp12}.

\begin{proof}
According to Lemma~\ref{lem:3_3} we have $N=T V_1$. Let $\Lambda_0$ be the
$V_1 E$-equivariant \general extension map with respect to $H_1\lhd V_1$ from
Theorem~\ref{thm:very_good_twist}. We obtain an $NE$-equivariant extension map
$\Lambda$ by sending $\la \in \Irr(T)$ to the common extension of $\la$ and
$\restr\Lambda_0(\restr \la|{H_1})|{ V_{1,\la}}$.
According to the proof of \cite[Lemma~4.3]{Sp09}, $\Lambda$ is then
well-defined.

For proving (2) let $\la\in\Irr(T)$. Then $\la_0\deq\restr\la|{H_1}$
extends to some $\w\la_0\in\Irr((V_1 E)_{\la_0})$ with $\w\la_0(v\wh F)=1$
by Lemma \ref{lem:3_13}. According to the proof of \cite[Lemma~4.3]{Sp09} there
exists a unique common extension $\w\la$ of $\Lambda(\la)$ and
$\restr \w\la_0| {(V_1E)_\la}$ to $(NE)_\la$. Then $\w\la(v\wh F)=1$.
\end{proof}

For later use we describe the action of $\w N E$ on the extension map $\Lambda$
from Corollary~\ref{cor:ker_delta_sc}. Recall
$\w T:=\w\bT^{vF}$ and $\w N:=\NNN_{\w\bG}(\bS)^{vF}$.  In the following we
set $W(\la):=N_\la/T$ for $\la\in\Irr(T)$ and
$W(\w\la):= N_\wla/T$ for $\w\la\in\Irr(\w T)$.

\begin{prop}   \label{prop:3_12}
 Let $\la\in\Irr(T)$, $\w\la \in \Irr(\w T|\la)$, $x \in \w N E$,
 and $\Lambda$ the extension map from Corollary~\ref{cor:ker_delta_sc}.
 Then the character $\delta\in\Irr(W(\la)^x)$ with
 $\delta\Lambda(\la^x)=\Lambda(\la)^x$ satisfies $\ker(\delta)\geq W(\w\la^x)$.
\end{prop}

\begin{proof}
Observe that $\delta$ is well-defined by \cite[Cor.~6.17]{Isa}.
Since $\Lambda$ is $NE$-equivariant $\delta$ associated with $x$
is trivial whenever $x\in NE$.
For $x \in \w T$ we have $\la^x=\la$ so $\delta$ has the stated property.
Taking those two results together we obtain the claim.
\end{proof}

As mentioned earlier the extension map constructed above is key to a labelling
and understanding of the characters of $\Irr(N)$.

\begin{prop}   \label{prop:5_11_here}
 Let $\Lambda$ be the extension map from Corollary \ref{cor:ker_delta_sc}
 with respect to $T\lhd N$. 
 Then the map
 \[\Pi:\cP=\{(\la,\eta)\mid \la\in\Irr(T),\,\eta\in\Irr(W(\la))\}
   \longrightarrow\Irr(N),\quad (\la,\eta)\longmapsto (\Lambda(\la)\eta)^{N},\]
 is surjective and satisfies
 \begin{enumerate}[label=\rm(\arabic*),ref=\thethm(\arabic{enumi})]
  \item $\Pi(\la,\eta)=\Pi(\la',\eta')$ if and only if there exists some
   $n\in N$ such that $\tw n\la=\la'$ and $\tw n\eta=\eta'$.
  \item $\tw \si\Pi(\la,\eta)=\Pi(\tw \si\la,\tw \si\eta)$ for every
   $\si\in E$.
  \item \label{param_diag} \label{prop:5_11_here3}
   Let $t\in\w T$, and $\nu_t\in\Irr(N_\la)$ be  the linear character
   given by $\tw t\Lambda(\la)=\Lambda(\la)\nu_t$.
   Then $N_{\w\la}=\ker(\nu_t)$ for any $\w\la\in \Irr(\spann<T,t>|\la)$.
   For $\w\la_0 \in \Irr(\w T|\la)$ the map
   $\w T \rightarrow\Irr(N_\la/N_{\w\la_0})$ given by $t\mapsto \nu_t$ is
   surjective, and $\tw t\Pi(\la,\eta)=\Pi(\la,\eta\nu_t)$.
 \end{enumerate}
\end{prop}

\begin{proof}
The arguments from \cite[Prop.~5.11]{CS15} can be transferred to prove the
statement. Straightforward considerations show that the map
in~\ref{param_diag} is surjective.
\end{proof}

\subsection{Maximal extendibility with respect to $W(\wla)\lhd W(\la)$}   \label{sec:maxextWwla}
Our aim in this subsection is to show that maximal extendibility holds with
respect to $W(\wla)\lhd \NNN_{W_1E}(W(\w\la))$ for every $\wla\in\Irr(\wT)$
with $W_1\deq \pi(N)$, where $W(\w\la)\deq N_\wla/ T$. Two less general
results are known in particular cases: For $d=1$ maximal extendibility is
known to hold with respect to $W(\wla)\lhd W(\la)$ where $\la=\restr \w\la|T$,
see Proposition \ref{prop:max:ext:Wla} below. Proposition~5.12 of \cite{CS15}
shows the analogue for arbitrary positive integers $d$ assuming that the
underlying root system is of type $\tA_l$.

The statement in Theorem~\ref{thm:3:25} plays a crucial role in proving
Theorem~\ref{thm:stab} via the parametrisation of characters of $N$ given above.

We start by rephrasing the old result known for $d=1$.

\begin{prop}   \label{prop:max:ext:Wla}
 Assume that $d=1$. Let $\la\in\Irr(T)$ and $\w\la\in\Irr(\w T|\la)$.
 Then maximal extendibility holds with respect to
 $W(\w\la)\lhd W(\la)$.
\end{prop}

\begin{proof}
The quotient $W(\la)/W(\w\la)$ is abelian and for every $\eta\in\Irr(W(\la))$
every character $\eta_0\in \Irr(W(\w\la)\mid \eta)$ has multiplicity one in
the restriction $\restr \eta|{W(\w\la)}$, see \cite[13.13(a)]{Cedric}.
This implies the statement.
\end{proof}

\begin{thm}   \label{thm:3:25}
 Let $\la\in\Irr( T)$, $\w\la\in\Irr(\w T|\la)$ and $W_1\deq \pi(N)$.
 Then every $\eta_0\in\Irr(W(\w\la))$ has an extension
 $\kappa\in\Irr(\norm{W_1 E}{W(\wla)}_{\eta_0})$ with $v\wh F\in \ker(\kappa)$.
\end{thm}

\begin{proof}
We first prove that maximal extendibility holds with respect to
$W(\wla)\lhd \NNN_{W_1 E}(W(\wla))$. Let $(\w \bG^*,\w \bT^*, v' F^*)$ be the
dual to $(\w\bG,\w\bT,vF)$ constructed as in \cite[Def.~13.10]{DM}.

Note that because of our particular choice of $v$ the automorphism on $W$
induced by $vF$ coincides with a graph automorphism $\phi'$ on $W$. The
character $\w\la$ corresponds to a semisimple element $s\in(\w \bT^*)^{v'F^*}$
of the dual group $(\bG^*,v'F^*)$.
Let $R(\wla)$ be the Weyl group of $\Cent_{\w\bG^*}(s)$.
Since $\Cent_{\w\bG^*}(s)$ is connected, $R(\w\la)$ is a reflection group.
We have $W(\wla)=\Cent_{R(\wla)}(v'F^*)=\Cent_{R(\w\la)}(\phi')$. Accordingly
$W(\wla)$ is a reflection group and the $\phi'$-orbits on the roots
of $R(\wla)$ form a root system, which we denote by $\Phi(\la)$, see
\cite[Thm.~C.5]{MT}. Straightforward calculations show that $\Phi(\la)$ is
already determined by $\la$.

The group $K:=\NNN_{W_1 E}(W(\w\la))$ acts on $W(\w\la)$, $R(\wla)$ and
$\Phi(\la)$ by conjugation. Let $\Delta$ be a base of $\Phi(\la)$. Then by
the properties of root systems $K=W(\w\la)\Stab_K(\Delta)$, even
$K=W(\w\la)\rtimes \Stab_K(\Delta)$, where $\Stab_K(\Delta)$ denotes the
stabiliser of $\Delta$ in $K$.

First let us prove that maximal extendibility holds with respect to
$W(\w\la)\rtimes \Aut(\Delta)$, where $\Aut(\Delta)$ is the group of
length-preserving automorphisms of $\Delta$.

Whenever $\Delta$ is indecomposable the statement is
true, since then all Sylow subgroups of $\Aut(\Delta)$ are cyclic.
If $\Delta=\Delta_1\sqcup\ldots\sqcup\Delta_r$ with isomorphic indecomposable
systems $\Delta_i$, the group $\Aut(\Delta)$ is isomorphic to the wreath
product $\Aut(\Delta_1)\wr \Sym_r$. Since maximal extendibility holds with
respect to $H^r\lhd H\wr \Sym_r$ for any group $H$ according to
\cite[Thm.~25.6]{Hu} we see that maximal
extendibility holds with respect to $W(\wla)\rtimes \Aut(\Delta)$ in that case.
Since maximal extendibility holds for $H_1\times H_2\lhd G_1\times G_2$
whenever it holds for $H_1\lhd G_1$ and $H_2\lhd G_2$ the above implies that
maximal extendibility holds with respect to
$W(\wla) \lhd W(\wla)\rtimes \Aut(\Delta)$.

Now let $C:=\Cent_K(\Delta)$. By definition $C\lhd K$.
Let $\o K:=K/C$. Then maximal extendibility holds with respect to
$W(\wla)\lhd K$ if it holds with respect to $\o R:= W(\wla)C/C\lhd  \o K$.
We see that $\o S:=\Stab_K(\Delta)/C$ is a subgroup of $\Aut(\Delta)$ and by
the above maximal extendibility holds with respect to
$W(\wla)\lhd W(\wla)\rtimes \o S$. But this implies maximal extendibility with
respect to $W(\wla)\lhd K$. This proves the first part of the claim.

We finish by constructing the required extension $\kappa$.
Let $\eta_0\in\Irr(W(\wla))$. Recall $E_1$ from~\ref{not:3.4}.
By the above $\eta_0$ extends to some $\wh F_0$-stable
$\kappa_1\in\Irr(\NNN_{W_1E_1}(W(\wla))_{\eta_0})$.
Since $\pi(v\wh F \wh F_0^{-m})\in \bZ(W_1E)$  and
$\NNN_{W_1E}(W(\wla))_{\eta_0}= \NNN_{W_1E_1}(W(\wla))_{\eta_0}
  \times \langle\wh F_0\rangle$ the considerations from
the proof of Lemma~\ref{lem:3_13} ensure the existence of $\kappa$, as required.
\end{proof}

\subsection{Consequences}\label{subsec:3:stab}
The previous considerations allow us also to conclude that the considered
characters of $N$ have the structure stated in Proposition \ref{prop:5_6}.

\begin{thm}   \label{thm:stab}
 For every $\chi\in\Irr(\w N)$ there exists some $\chi_0\in\Irr(N|\chi)$
 with the following properties:
 \begin{enumerate}[label=\rm(\arabic*),ref=(\arabic{enumi})]
  \item $(\w N \rtimes E)_{\chi_0} =\w N_{\chi_0} \rtimes E_{\chi_0}$;
   and
  \item $\chi_0$ has an extension $\w\chi_0\in\Irr(N\rtimes E_{\chi_0})$
   with $v\wh F\in\ker(\w \chi_0)$.
  \end{enumerate}
\end{thm}

\begin{proof}
Let $\chi_1\in\Irr(N|\chi)$ and $(\lambda,\eta)\in\cP$ with
$\chi_1=\Pi(\la,\eta)$ for the map $\Pi$ from Proposition~\ref{prop:5_11_here}.

Let $\w\la\in\Irr(\wT|\la)$ and $\eta_0\in\Irr(W(\w\la))$ such that
$\eta\in\Irr(W(\la)|\eta_0)$. By Clifford correspondence there exists a unique
character $\eta_1\in\Irr(W(\la)_{\eta_0}|\eta_0)$ such that
$\eta=\eta_1^{W(\la)}$. Now since $W(\la)/W(\w\la)$ is abelian and as by
Proposition~\ref{prop:max:ext:Wla} maximal extendibility holds with respect
to $W(\w\la)\lhd W(\la)$, the character $\eta_1$ is an extension of $\eta_0$.

Let $W_1\deq \pi(N)$. According to Theorem \ref{thm:3:25} there exists an
$\NNN_{W_1E}(W(\w\la))_{\eta_0}$-invariant extension
$\w\eta_0\in\Irr(W(\la)_{\eta_0})$ of $\eta_0$.
The character $\eta':=(\w\eta_0)^{W(\la)}$ is irreducible.
Hence $\chi_0:=\Pi(\la,\eta')$ is a well-defined character of $N$.

We show that $\chi_0$ is $\wN$-conjugate to $\chi_1$: Since the map
$\wT\rightarrow\Irr(W(\la)/W(\w\la))$, $t\mapsto\nu_t$, from
Proposition~\ref{prop:5_11_here}(3) is surjective, there exists $t\in \w T$
such that $\eta'=\eta\nu_t$. This proves
$\tw t\chi= \tw t \Pi(\la,\eta)= \Pi(\la,\eta\nu_t)= \chi_0$.

For analysing the stabiliser of $\chi_0$ let $t\in\wT$ and $e\in E$ such
that $\chi_0^{te}=\chi_0$. Then there exists some $n\in N$ such that
$(\la,\eta')= (\la^{ne},(\eta')^{ne} \nu_t)$. Without loss of generality $n$
can be chosen such that $\pi(n)e\in\NNN_{W_1E}(W(\w\la))_{\eta_0}$.
By the choice of $\w\eta_0$, $\pi(n)e$ stabilises $\w\eta_0$, hence
$(\la^{ne}, (\eta')^{ne})=(\la, \eta')$ and $\chi_0^{e}=\chi_0$.
This proves the equation in (1).

By Corollary~\ref{cor:ker_delta_sc} the character $\Lambda(\la)$ has an
extension $\w\la$ to $(NE)_{\Lambda(\la)\eta'}$ with $\w\la(v\wh F)=1$.
On the other hand the character $\w\eta_0$ can be chosen to have an extension
$\wh \eta_0$ to $\NNN_{W_1E}(W(\w\la))_{\eta_0}$ with
$v\wh F \in \ker(\wh \eta_0)$, see Theorem~\ref{thm:3:25}.
We denote by $\kappa_1$ the lift of
$\restr\wh \eta_0|{\NNN_{W_1E}(W(\wla))_{\eta_0,\la}}$ to $(NE)_{\eta_0,\la}$.
Then $\kappa_2:=(\kappa_1)^{(NE)_{\la,\eta}}$ is irreducible with
$v\wh F\in \ker(\kappa_2)$. The character $(\wla \kappa_2)^{(NE)_{\chi_0}}$
is an extension of $\chi_0$ with the required properties.
\end{proof}

Via Proposition~\ref{prop:5_6} the above proves Theorem~\ref{thm:IrrN_autom}.
For later two further consequence of our considerations are important.
First we give the following interpretation of Theorem~\ref{thm:IrrN_autom}.

\begin{lem}   \label{lem:3_21}
 Let $\bS_0$, $N_0$, $\w N_0$ and $O_0$ be defined as in
 Theorem~\ref{thm:IrrN_autom}. For $\psi_0\in\Irr(N_0)$ the following are
 equivalent:
  \begin{enumerate}[label=\rm(\roman*),ref=(\roman{enumi})]
   \item \label{lem3:2i}
   $O_0= (\wGF\cap O_0) \rtimes (D\cap O_0)$.
   \item \label{lem3:2ii}
   $(\w G^F D)_{\bS_0,\psi_0}=\w N_{0,\psi_0}(\bG^F\rtimes D)_{\bS_0,\psi_0}$.
  \end{enumerate}
\end{lem}

\begin{proof}
By the definition of $O_0$ one deduces~\ref{lem3:2ii} from~\ref{lem3:2i}
by considering the stabiliser of $\bS_0$:
\begin{align*}
(\w G^F D)_{\bS_0,\psi_0}&= (O_0)_{\bS_0}
   =\left((\wGF\cap O_0) \rtimes (D\cap O_0)\right )_{\bS_0}=\\
  &=\left((\wGF\cap O_0) (\GF D\cap O_0)\right )_{\bS_0}
   =(\wGF\cap O_0)_{\bS_0} (\GF D\cap O_0)_{\bS_0}=\\
  &=\wN_{0,\psi_0} (\GF D)_{\bS_0,\psi_0}.
\end{align*}
Here, recall that by \cite[Thm.~3.4]{BM92} all Sylow $d$-tori of $(\bG,F)$ are
$\GF$-conjugate and the $(D\cap O_0)$-conjugates of $\bS_0$ are Sylow $d$-tori.

Multiplying the equation in \ref{lem3:2ii} with $\GF$ gives
$O_0=(\w \bG^F\cap O_0) ( (\GF \rtimes D) \cap O_0)$.
Since $\GF\leq (\GF \rtimes D) \cap O_0 $ this gives~\ref{lem3:2i}.
\end{proof}

\begin{prop}   \label{wLam_sec3}
 Let $\bS_0$, $N_0$ and $\w N_0$ be defined as in Theorem~\ref{thm:IrrN_autom}.
 Let $\w C_0\deq \Cent_{\w\bG^F}(\bS_0)$. Then there exists some
 $\NNN_{\w\bG^FD}(\bS_0)$-equivariant \general extension map $\w\Lambda$ with
 respect to $\w C_0 \lhd \w N_0$, such that in addition
 $\w\Lambda(\w\la\restr\delta|{\w C_0})=\w\Lambda(\w\la)\restr\delta|{\w N_0}$
 for every $\w\la\in\Irr(\w C_0)$ and $\delta\in\Irr(\w\bG^{F}|1_{\bG^{F}})$.
 \end{prop}

\begin{proof}
The considerations from the proof of \cite[Cor.~5.14]{CS15} can be transferred:
Applying the isomorphism $\iota$ from Proposition~\ref{prop:5_6} shows that
it is sufficient to verify that there exists some $NE $-equivariant \general
extension map $\w\Lambda$ with respect to $\w T\lhd \w N$, such that in
addition $\w\Lambda(\w\la\restr\delta|{T})=\w\Lambda(\w\la)\restr\delta|{\w N}$
for every $\w\la\in\Irr(\w T)$ and $\delta\in\Irr(\w\bG^{vF}|1_{\bG^{vF}})$.
Let $\Lambda$ be the $ N E$-equivariant extension map with respect
to $T \lhd  N$ from Corollary~\ref{cor:ker_delta_sc} and
$$\w\Lambda:\Irr(T) \rightarrow \bigcup_{T\leq I\leq N}\Irr(I)$$
be the map sending $\w\la\in\Irr(\w T )$ to the unique common
extension of $\w\lambda$ and $\restr\Lambda(\la)|{N_\wla}$ where $\la\deq\restr\wla|T$.
Then $\w\Lambda$ is well-defined according to \cite[Lemma~4.3]{Sp09} and has
the required properties.
\end{proof}

The next statement is later applied to verify
assumption~\ref{hauptprop_maxext_loc} in the considered cases.

\begin{cor}   \label{cor:maxextNwN}
 For the groups $N_0$ and $\w N_0$ from Theorem~\ref{thm:IrrN_autom} maximal
 extendibility holds with respect to $N_0\lhd \w N_0$.
\end{cor}

\begin{proof}
Like in the proof of the preceding proposition the isomorphism $ \iota$ from
the proof of Proposition \ref{prop:5_6} allows us to prove the statement by
establishing that maximal extendibility holds with respect to $N\lhd \w N$.
Let $\w\Lambda$ be the \general extension map with respect to $\w T \lhd \w N$
from (the proof of) Proposition~\ref{wLam_sec3}. Then every character
$\w\psi\in \Irr(\w N)$ is of the form $(\w\Lambda(\w\la)\eta_0)^{\w  N}$ for
some $\w\lambda\in\Irr(\w T )$ and $\eta_0\in\Irr(W(\wla))$. Thus
\begin{eqnarray*}
\restr\w\psi|{ N }= \restr (\w\Lambda(\w\la)\eta_0)^{\w  N }|{ N }&=&
  \Big (\restr(\w\Lambda(\w\la)\eta_0)|{ N _\wla}\Big)^{  N }=
  \Big (\restr\Lambda(\la)|{ N _\wla}\eta_0\Big)^{ N }=\\
  &=&\Big((\restr\Lambda(\la)|{ N _\wla}\eta_0)^{ N _\la}\Big)^{ N }=
  \Big (\Lambda(\la)(\eta_0^{ N _\la})\Big)^{ N },
\end{eqnarray*}
where $\la:=\restr \w\la|{ T }$. According to Theorem~\ref{thm:3:25} maximal
extendibility holds with respect to $W(\wla)\lhd W(\la)$. Since
$W(\la)/W(\wla)$ is abelian, $\eta_0^{N_\la}$ and hence $\restr \w\psi|N$
is multiplicity-free. This proves the statement.
\end{proof}

\section{Some non-principal series in symplectic groups}   \label{type C}

\noindent
For later applications in the study of some non-principal Harish-Chandra
series in type $\tC_l$ we have to consider 
characters of a certain standard Levi subgroup that is not a torus.

In this section let $\bG=\rC_{l,\SC}$, so $G=\Sp_{2l}(q)$ and
$\tilde G=\CSp_{2l}(q)$, and let $L$ be a standard Levi subgroup of $G$ with
root system of type $\tC_1$. Then for the $F$-stable torus
$\bT_0:=\Cent_\bG^\circ(L)$ we have $L=\Cent_G(\bT_0)$.
Additionally let $N=\NNN_G(\bT_0)$, and $\w N=\NNN_{\w G}(\bT_0)$.

Let $\{\al_1,\ldots,\al_l\}$ be the base of the root system of $\bG$ as
introduced in \cite[6.1]{Spaeth2}. Note that
$L=\spann<T,X_{\al_1},X_{-\al_1}>$, and $L_0:=\spann<X_{\al_1},X_{-\al_1}>$
is isomorphic to $\SL_2(q)$. The group $\w T:=\w\bT^F$ induces diagonal
automorphisms on $\SL_2(q)$ and $F_0$ induces the field automorphism.

Let $T_1:=\spann<h_{e_i}(t) \mid i \geq 2\,,\, t\in\FF_q^\times>$. Then an
easy calculation gives that $L=T_1\times L_0$ and $N=N_1\times L_0$ where
$N_1:=\spann<n_{e_i}(t),n_{\pm e_i \pm e_j}(t)\mid
  2\leq i<j\leq l\,,\, t\in\FF_q^\times>$.
Let $\w L:=\w T L$. Note that $D$ acts on $N_1$ and on $L_0$. The diagonal
automorphism induced by $\tilde G$ acts as diagonal automorphism on $N_1$
and on $L_0$.

\begin{prop}   \label{prop:ext_map_C}
 There exists an $ND$-equivariant \general extension map with respect
 to $L\lhd N$.
\end{prop}

\begin{proof}
Straightforward computations show that $T_1\lhd N_1$ are a maximally split
torus and its normaliser in the group $\spann<X_{e_i},X_{\pm e_i \pm e_j}
\mid 2\leq i<j\leq l>$, which is isomorphic to $\tC_{l-1,\SC}(q)$.
Let $\Lambda_1$ be the \general extension map with respect to $T_1\lhd N_1$ for
$\Irr(T_1)$ from Corollary~\ref{cor:ker_delta_sc}.

Any character $\psi\in\Irr(L)$ has the form as $\la_1\times \zeta$ with
$\la_1\in\Irr(T_1)$ and $\zeta\in\Irr(L_0)$. The stabiliser $N_\psi$ coincides
with $N_{1,\la_1}\times L_0$. Accordingly we can define an extension of $\psi$
as $\Lambda_1(\la_1) \times \zeta$. 
Then
\begin{equation}\label{eq_def_Lambda_C}
  \Lambda:\Irr(L) \rightarrow \bigcup_{L\leq I \leq N} \Irr(I),\quad
  \la_1\times \zeta \mapsto\Lambda_1(\la_1) \times \zeta,
 \end{equation}
is an extension map as required. Now since $D$ induces field automorphisms on
$N_1$, $\Lambda_0$ and hence $\Lambda$ are $D$-equivariant. The action
of $N$ on $\Irr(I)$ for subgroups $I$ with $L\leq I \leq N$ coincides with
the one of $N_1$. This implies by definition that $\Lambda$ is
$ND$-equivariant, since $\Lambda_1$ is $N_1D$-equivariant.
\end{proof}

As in Section~\ref{sec:IrrlN} the extension map from
Proposition~\ref{prop:ext_map_C} can
be used to give a labelling to the characters of $N$ lying above cuspidal
characters of $L$. We write $\Irr_\cu(L)$ for the set of cuspidal characters
of $L$ and $\Irr_\cu(N)\deq \Irr(N|\Irr_\cu(L) )$.
For $\la\in\Irr(L)$ set ${W(\la)}\deq N_\la/L$.

\begin{prop}   \label{prop:loc_param_C}
 Let $\Lambda$ be the extension map with respect to $L\lhd N$ from
 Proposition~\ref{prop:ext_map_C}.
 Then
 \[\Pi:\cP=\{(\la,\eta)\mid\la\in\Irr_\cu(L),\,\eta\in\Irr({W(\la)})\}
   \longrightarrow\Irr_\cu(N), \quad
   (\la,\eta)\longmapsto (\Lambda(\la)\eta)^{N},\]
 is surjective and satisfies
 \begin{enumerate}[label=\rm(\arabic*),ref=(\arabic{enumi})]
  \item $\Pi(\la,\eta)=\Pi(\la',\eta')$ if and only if there exists some
   $n\in N$ such that $\tw n\la=\la'$ and $\tw n\eta=\eta'$.
  \item $\tw \si\Pi(\la,\eta)=\Pi(\tw \si\la,\tw\si\eta)$ for every $\si\in D$.
  \item Let $t\in\w L_\la$, and $\nu_t\in\Irr(N_\la)$ the linear character
   given by $\tw t\Lambda(\la)=\Lambda(\tw t\la)\nu_t$.
   Then $N_{\w\la}=\ker(\nu_t)$ for any $\w\la\in \Irr(\spann<L,t>|\la)$. Let
   $\w\la_0$ be an extension of $\la$ to $\w L_\la$. Then the associated map
   $\w L_\la \rightarrow\Irr(N_\la/N_{\w\la_0})$, $t\mapsto\nu_t$, is
   surjective and $\tw t\Pi(\la,\eta)=\Pi(\la^t,\eta\nu_t)$.
 \end{enumerate}
\end{prop}

\begin{proof}
The proof of Proposition~\ref{prop:5_11_here} can be transferred.
\end{proof}

\begin{cor}
 Let $\la\in\Irr(L)$, $\w\la\in \Irr(\w L|\la)$ and $x\in \w N D$. Then
 $\delta_{\la,x}$ defined by $\delta_{\la,x}\Lambda(\tw x\la)
  =\tw x\Lambda(\la)$ satisfies
 $W(\tw x \w\la)\leq \ker(\delta_{\la,x})$.
\end{cor}

\begin{proof}
Using the result from the previous proposition the considerations
proving Proposition~\ref{prop:3_12} imply the statement.
\end{proof}

\begin{thm}   \label{thm:stab_C}
 Every character $\psi\in\Irr(N)$  satisfies
 $(\w N \rtimes D)_{\psi} =\w N_{\psi} (ND)_{\psi}$.
\end{thm}

\begin{proof}
Since $N=N_1\times L_0$ and $\w ND$ stabilises $N_1$ and $L_0$ we obtain that
$(\w ND)_{\psi}= (\w N D)_{\chi}\cap (\w N D)_{\zeta}$ for
$\psi=\chi\times\zeta$ with $\chi\in\Irr(N_1|\psi)$ and
$\zeta\in\Irr(L_0|\psi)$.

By direct calculations for $\SL_2(q)$, $\xi$ satisfies
$(\w N D)_{\xi}=\w N_\xi \rtimes D_\zeta$, since $\w N$ induces diagonal
automorphisms of $L_0$ and $D$ field automorphisms. Following
Theorem~\ref{thm:IrrN_autom} together with Lemma~\ref{lem:3_21} every character
$\chi\in\Irr(N_1)$ satisfies $(\w N D)_\chi=(\w N)_\chi D_\chi$. Together
with the above this implies the claim.
\end{proof}

\section{The action of $\Aut(G)$ on Harish-Chandra induced characters}   \label{sec:HC}
\noindent
The aim of this section is to verify that assumptions of
Theorem~\ref{thm:Sp12} concerning the characters of $G$ are satisfied.
For this we describe the action of $\Aut(G)$ on Harish-Chandra induced
characters in terms of their parameters. Thus we first have to recall how one
obtains the parametrisation of those characters. We follow here the treatment
of the subject given in \cite[Chap.~10]{Ca85}, which is based on the
results of \cite{HL} and \cite{HL83}.

We consider the following slightly more general setting. Let $G$ be a finite
group with a split $BN$-pair of characteristic~$p$. We write $W=N/(N\cap B)$
for the Weyl group of $G$, which we assume to be of crystallographic type.
Then there is a root system $\Phi$ attached to $W$ and we let $\Delta$ denote
a base of $\Phi$ corresponding to the simple reflections of $W$. We write
$s_\al\in W$ for the reflection along the root $\al\in\Phi$.

Let $P\le G$ be a standard parabolic subgroup with standard Levi subgroup $L$
and Levi decomposition $P= U\rtimes L$. Let $N(L):=(N_G(L)\cap N)L$. We choose
and fix once and for all an $N(L)$-equivariant extension map for $L\lhd N(L)$,
which exists according to \cite{GeckHC} and \cite[Thm.~8.6]{Lu}.

Let $\la$ be an irreducible cuspidal character of $L$. Via the Levi
decomposition $\la$ can be inflated to a character of $P$. Let $M$ a left
$\CC P$-module affording $\la$ and denote
by $\rho$ the corresponding representation. Let $\fF(\rho)$ be the vector
space of $\CC$-linear maps $f:\CC G\rightarrow M$ with
\[f(px)=\rho (p) f(x) \quad\text{ for all }p \in P \text{ and } x \in \CC G.\]
This vector space becomes a $\CC G$-module via
\begin{equation}\label{G_action}
  (g\star f)(x)=f(xg) \quad\text{ for all $g\in G$, $f\in \fF(\rho)$ and
  $x\in \CC G$}.
\end{equation}
We denote by $\R_L^G(\la)$ the character of $G$ afforded by this module.
It is known that $\R_L^G(\la)$ only depends on $\la$, not on the choice of $P$
or of $\rho$. The set of constituents of $\R_L^G(\la)$ is called the
Harish-Chandra series above $(L,\la)$ and will be denoted by $\cE(G,(L,\la))$.
The union of Harish-Chandra series associated with $N\cap B$ and its
characters is called the \emph{principal series of} $G$. 

\subsection{Actions of automorphisms on the standard basis}
Let $\si$ be an automorphism of $G$ stabilising $P$, $L$, and the $BN$-pair.
Recall that $\si$ acts on the class functions on $G$ via
$\chi\mapsto{}^\si\chi$, where $^\si\chi(g)=\chi(\si^{-1}(g))$ for all $g\in G$.
It is immediate from the definitions that $(L,\sila)$ is again a cuspidal pair
of $G$.

The character $^\si\R_L^G(\la)$ is afforded by the $\CC G$-module
$^\si\fF (\rho)$ obtained from the vector space $\fF(\rho)$ together with
the $G$-action
\begin{equation}\label{G_action_si}
  (g\star_{\si}f)(x)= f(x\si^{-1}(g)) \quad\text{ for all $g\in G$,
  $f\in \fF(\rho)$, and $x\in \CC G$}.
\end{equation}
One easily sees that $\End_{\CC G}(\fF(\rho))$ and $\End_{\CC G}(^\si\fF(\rho))$
can be canonically identified via ${}^\si\! B(f):=B(f)$
for $B\in \End_{\CC G}(\fF(\rho))$ and $f\in\fF(\rho)$. Let $\fF(\sirho)$ be
the coinduced module associated to $\sirho$ defined as above. Then
$\iota:{}^\si\fF(\rho) \rightarrow \fF(\sirho)$ given by
\[ f\mapsto {}^\si\!f\text{ with } {}^\si\!f(x)=f(\si^{-1}(x))
  \text{ for all } x \in \CC G\]
defines a $\CC G$-module isomorphism. Moreover
$B\mapsto\iota\circ B\circ\iota^{-1}$ for $B\in\End_{\CC G}(^\si\fF(\rho))$
induces an isomorphism from $\End_{\CC G}(^\si\fF(\rho))$ to
$\End_{\CC G}(\fF(\sirho))$. We denote by
$\wi:\End_{\CC G}(\fF(\rho))\rightarrow\End_{\CC G}(^\si\fF(\rho))\rightarrow
\End_{\CC G}(\fF(\sirho))$ the composed isomorphism.

Since we are interested in the irreducible constituents of $\R_L^G(\la)$ and
$\R_L^G(\sila)$, which are parametrised by the isomorphism classes of
irreducible modules of $\End_{\CC G}(\fF(\rho))$ and of
$\End_{\CC G}(\fF(\sirho))$ respectively, see \cite[Prop.~10.1.2]{Ca85}, we
will need to compute $\wi(B)$ for some elements $B\in\End_{\CC G}(\fF(\rho))$.
We start by determining $\wi$ on a natural basis of $\End_{\CC G}(\fF(\rho))$.

With $N(L)=(N_G(L)\cap N)L$ let $W_G(L):=N(L)/L$, the \emph{relative Weyl
group} of $L$ in $G$, and set $W(\la):=N(L)_\la/L$. For $w\in W(\la)$ we
denote by $\dot{w}\in N(L)$ a once and for all chosen preimage under the
natural map. We let $\Phi_L\subseteq\Phi$ denote the root system of $W_L$,
with simple system $\Delta_L\subseteq\Delta$.

Let $\w\rho$ be an extension of $\rho$ to $N(L)_\la$ affording the extension
$\Lambda(\la)$ from our chosen equivariant extension map $\Lambda$. For
$w\in W_G(L)$ let ${\rB_{w,\rho}\in\End_{\CC G}(\fF(\rho))}$ be defined by
\[ (\rB_{w,\rho} f) (x)= \w \rho(\dot{w}) f(\dot{w}^{-1}e_U x )\quad
  \text{ for all }f\in \fF(\rho)\text{ and } x \in \CC G,\]
where $e_U:=\frac 1{|U|}\sum_{u\in U} u$ is the idempotent associated to the
unipotent radical $U$ of $P$. Note that $\rB_{w,\rho}$ is independent of the
actual choice of $\dot{w}$.

Analogously we define $\rB_{w,\sirho}$ by using the extension $\w\rho'$ of
$\sirho$ affording $\Lambda(\sila)$. Note that $\w\rho'$ and $^\si(\w\rho)$
then may differ. We denote by $\delta_{\la,\si}\in \Irr(W(\sila))$ the
character of
$N(L)_{\sila}$ with $\delta_{\la,\si}\Lambda(\sila)={}^\si\Lambda(\la)$.
This character is well-defined by \cite[Cor.~6.17]{Isa}.

For $w\in W(\la)$ let ${\rB_{w,\sirho}\in \End_{\CC G}(\fF(\sirho))}$ be
defined via
\[ (\rB_{w,\sirho} f) (x)= \w \rho'(\dot{w}) f(\dot{w}^{-1}e_U x)\quad
   \text{ for all }f\in \fF(\sirho)\text{ and } x \in \CC G.\]

\begin{lem}   \label{lem:wi(B)}
 For all $w\in W(\la)$ we have
 $\wi (\rB_{w,\rho})= \delta_{\la,\si}(\si(w))\,\rB_{\si(w),\sirho}$.
\end{lem}

\begin{proof}
Indeed, for $f\in\fF(^\si\rho)$ and $x \in \CC G$ we have
$$\wi(\rB_{w,\rho})(f)(x)
  =\w\rho(\dot{w})f(\sigma(\dot{w^{-1}}e_U\sigma^{-1}(x)),$$
which agrees with
$$\delta_{\la,\si}(\si(w))\,\rB_{\si(w),\sirho}(f)(x)
  = \w\rho(\dot{w})f(\sigma(\dot{w})^{-1}e_Ux)$$
as $\si(e_U)=e_U$.
\end{proof}

\subsection{The decomposition of $W(\la)$}   \label{subsec:p alpha}
In order to transfer our results to the $\T_{w,\rho}$-basis of the endomorphism
algebra we need to recall the semi-direct product decomposition of $W(\la)$,
see \cite[Sec.~2 and~4]{HL}. Define
$$\hat\Omega:=\{\al\in\Phi\setminus\Phi_L\mid
   w(\Delta_L\cup\{\al\})\subseteq\Delta\text{ for some }w\in W\},$$
and for $\al\in\hat\Omega$ set $v(\al):=w_0^L w_0^\al$, where
$w_0^L,w_0^\al$ are the longest elements in $W_L$,
$\langle W_L,s_\al\rangle$ respectively. Then let
$\Omega:=\{\al\in\hat\Omega\mid v(\al)^2=1\}$. Note that $\Omega$ is
$\si$-invariant. For $\al\in\Omega$ let $L_\al$ denote the standard Levi
subgroup of $G$ corresponding to the simple system $\Delta_L\cup\{\al\}$. Then
$L$ is a standard Levi subgroup of $L_\al$. We write $p_{\al,\la}\geq 1$
for the ratio between the degrees of the two different constituents of
$\R_L^{L_\al}(\la)$. Let
$$\Phi_\la:=\{\al\in\Omega\mid s_\al\in W(\la),\,p_{\al,\la}\neq 1\},$$
a root system with set of simple roots $\Delta_\la\subseteq\Phi_\la\cap\Phi^+$,
and let $R(\la):=\langle s_\al\mid \al\in\Phi_\la\rangle$ its Weyl group.
Then $W(\la)$ satisfies $W(\la)=R(\la)\rtimes C(\la)$, where the group
$C(\la)$ is the stabiliser of $\Delta_\la$ in $W(\la)$, see
\cite[Prop.~10.6.3]{Ca85}.

\begin{lem}   \label{lem:p}
 We have $p_{\al,\la}=p_{\si(\al),\sila}$ for all $\al\in\Phi_\la$ and
 hence $R(\sila)=\si(R(\la))$ and $C(\sila)=\si(C(\la))$.
\end{lem}

\begin{proof}
By definition we have
$^\si\R_L^{L_\alpha}(\la)= \R_L^{L_{\si(\al)} }(\sila)$ since $\si$
stabilises $U$. This implies $p_{\al,\la}=p_{\si(\al),\sila}$ by its
definition.
\end{proof}

For $w\in W$ we set $\ind(w):=|U_0\cap (U_0)^{w_0w}|$, where $U_0$ is the
unipotent radical of $B$ and $w_0\in W$ is the longest element. Also,
for $\al\in\Delta_\la$ a simple root of $\Phi_\la$ we define
$\eps_{\al,\la}\in \{\pm1\}$ by
\begin{equation}   \label{def_ep}
  \rB_{s_\al,\rho}^2= \ind(s_\al)\,\id + \eps_{\al,\la}
  \frac{p_{\al,\la}-1}{\sqrt{\ind(s_\al)p_{\al,\la} }} \rB_{s_\al,\rho}
\end{equation}
(see \cite[Prop.~10.7.9]{Ca85}). Here, the square root is always taken positive.

\begin{lem}   \label{lem:eps}
 If $R(\sila)\leq \ker(\delta_{\la,\si})$ then $\ind(\si(s_\al))=\ind(s_\al)$
 and $\eps_{\si(\al),\sila}=\eps_{\al,\la}$ for all $\al\in\Delta_\la$.
\end{lem}

\begin{proof}
Let $\al\in \Delta_\la$ and set $s:=s_\al$, $\al':=\si(\al)$, $s':=s_{\al'}$,
$\la':=\sila$, $\rho':=\sirho$. Applying $\wi$ to Equation~\ref{def_ep}
we obtain
\[\wi(\rB_{s,\rho}^2)=\ind(s)\,\id+\eps_{\al,\la}
  \frac{p_{\al,\la}-1}{\sqrt{\ind(s)p_{\al,\la}}}\wi(\rB_{s,\rho}).\]
Now $p_{\al',\la'}=p_{\al,\la}$ by Lemma~\ref{lem:p}, and
since $\si$ stabilises $U_0$ and $w_0$ we also have $\ind(s')=\ind(s)$.
Then Lemma~\ref{lem:wi(B)} yields
\[ \delta_{\la,\si}(s')^2\, \rB_{s',\rho'}^2=
  \ind(s')\,\id + \eps_{\al,\la} \frac{p_{\al',\la'}-1}
  {\sqrt{ \ind(s') p_{\al',\la'}}} \delta_{\la,\si}(s')\,\rB_{s',\rho'}.\]
Since $s'\in \si(R(\la))=R(\sila)$ the assumption
$R(\la')\leq \ker(\delta_{\la,\si})$ allows us to simplify this to
\[ \rB_{s',\rho'}^2=\ind(s')\,\id + \eps_{\al,\la}\frac{p_{\al',\la'}-1}
  {\sqrt{\ind(s')p_{\al',\la'}}}\,\rB_{s',\rho'}.\]
The claim follows by comparison with~\eqref{def_ep} for $\rB_{s',\rho'}$.
\end{proof}

Now for $\al\in\Delta_\la$ set $\T_{s_\al,\rho}:=
  \eps_{\al,\la}\,\sqrt{\ind(s_\al)p_{\al,\la}}\,\rB_{s_\al,\rho}$; for
$w\in R(\la)$ with a reduced expression $w=s_1\cdots s_r$
with $s_i=s_{\al_i}$ simple reflections (so $\al_i\in\Delta_\la$) let
$\T_{w,\rho}:=\T_{s_1,\rho}\cdots \T_{s_r,\rho}$; for
$w\in C(\la)$ define $\T_{w,\rho}:=\sqrt{\ind(w)}\,\rB_{w,\rho}$, and then
for $w\in W(\la)$ with $w=w_1w_2$ where $w_1\in C(\la)$, $w_2\in R(\la)$,
let $\T_{w,\rho}:=\T_{w_1,\rho}\T_{w_2,\rho}$. This does not depend on
the choice of reduced expressions, see \cite[Prop.~10.8.2]{Ca85}. Then we have:

\begin{prop}   \label{prop:wi(T)}
 If $R(\sila)\leq \ker(\delta_{\la,\si})$ then for all $w\in W(\la)$ we have
 $$\wi(\T_{w,\rho})=\delta_{\la,\si}(\si(w))\, \T_{\si(w),\sirho}.$$
\end{prop}

\begin{proof}
First assume that $w=s_\al=:s$ for some $\al\in\Delta_\la$. Then
$$\begin{aligned}
  \wi(\T_{s,\rho})
  &=\wi\big(\eps_{\al,\la}\,\sqrt{\ind(s)p_{\al,\la}}\,\rB_{s,\rho}\big)\\
  &=\eps_{\al,\la}\,\sqrt{\ind(s)p_{\al,\la}}\,\,\wi(\rB_{s,\rho})
  =\eps_{\al,\la}\,\sqrt{\ind(s)p_{\al,\la}}\,\,
   \delta_{\la,\si}(s')\rB_{s',\sirho}
\end{aligned}$$
by Lemma~\ref{lem:wi(B)}, where $s'=\si(s)$, $\al'=\si(\al)$.
From Lemmas~\ref{lem:p} and~\ref{lem:eps} we know $p_{\al',\la'}=p_{\al,\la}$,
$\ind(s')=\ind(s)$ and $\eps_{\al',\la'}=\eps_{\al,\la}$. So indeed
$$\wi(\T_{s,\rho})
  =\delta_{\la,\si}(s')\,\eps_{\al',\la'}\,\sqrt{\ind(s')p_{\al',\la'}}\,
  \rB_{s',\sirho}=\delta_{\la,\si}(s') \T_{s',\sirho}.$$
Next, if $w\in C(\la)$ then
$$\wi(\T_{w,\rho})=\sqrt{\ind(w)}\,\wi(\rB_{w,\rho})
  =\delta_{\la,\si}(\si(w))\sqrt{\ind(w)}\rB_{\si(w),\sirho}
  =\delta_{\la,\si}(\si(w))\,\T_{\si(w),\sirho}.$$
In the general case, let $w\in W(\la)$ with $w=w_1w_2$ where $w_1\in C(\la)$,
and $w_2\in R(\la)$ has a reduced expression $w_2=s_1\cdots s_r$. Then by the
above we get
$$\begin{aligned}
  \wi(\T_{w,\rho})
  &=\wi(\T_{w_1,\rho})\wi(\T_{s_1,\rho})\cdots \wi(\T_{s_r,\rho})\\
  &=\delta_{\la,\si}(\si(w_1))\Big(\prod_{i=1}^r\delta_{\la,\si}(\si(s_i))\Big)\,
    \T_{\si(w_1),\sirho}\T_{\si(s_1),\sirho}\cdots \T_{\si(s_r),\sirho}\\
  &=\delta_{\la,\si}(\si(w_1s_1\cdots s_r))\, \T_{\si(w_1),\sirho}\T_{\si(s_1\cdots s_r),\sirho}
  =\delta_{\la,\si}(\si(w))\, \T_{\si(w),\sirho}
\end{aligned}$$
as claimed.
\end{proof}

\subsection{Central-primitive idempotents of $\End_{\CC G}(\fF(\rho))$}
Next we describe the central-primitive idempotents of $\End_{\CC G}(\fF(\rho))$.
It is well-known (see e.g.~\cite[Prop.~10.9.2]{Ca85}) that
$\End_{\CC G}(\fF(\rho))$
is a symmetric algebra with symmetrising trace defined by the linear map
$\tau_\rho:\End_{\CC G}(\fF(\rho))\rightarrow \CC$ with
\[ \tau_\rho(\T_{w,\rho})=\begin{cases} 1 & w=1, \\
                                        0& w\neq 1.\end{cases}\]
Let us denote  by $\{\T_{w,\rho}^\vee\}$ the basis dual to
$\{\T_{w,\rho}\}$ with respect to the bilinear form associated with
$\tau_\rho$. Thus
$$\T^\vee_{w,\rho}=p_{w,\la}^{-1} \T_{w^{-1},\rho}\qquad
  \text{for }w\in W(\la)$$
where $p_{w,\la}:=\prod_{\al\in \Phi_\la^+, \,\, w(\al)<0} p_{\al,\la}$
(see \cite[p.~349]{Ca85} or \cite[8.1.1]{GP}).
Note that via $\wi$, $\tau_\rho$ defines a symmetrising trace
$\tau_{\sirho}$ on $\End_{\CC G}(\fF(\sirho))$ with
\[ \tau_{\sirho}(\T_{w,\sirho})=\begin{cases} 1 & w=1, \\
                                                  0& w\neq 1.\end{cases}\]

Let $M$ be a simple $\End_{\CC G}(\fF(\rho))$-module and $\eta$ its
character. It can be considered as a submodule of $e_{\eta,\rho} \fF(\rho)$,
where
\[e_{\eta,\rho}:=\frac 1 {c_{\eta ,\rho}}
   \sum_{w\in W(\la)} \eta(\T_{w,\rho})\, \T_{w,\rho}^\vee \]
denotes the central-primitive idempotent of $\End_{\CC G}(\fF(\rho))$
corresponding to $M$ (see \cite[7.2.7(c)]{GP}). Here, $c_{\eta,\rho}$ is
the Schur element associated to $\eta$ as in \cite[Thm.~7.2.1]{GP}.

\begin{prop}   \label{prop:wi(e)}
 Assume that $R(\sila)\leq \ker(\delta_{\la,\si})$.
 There exists a simple $\End_{\CC G}(\fF(\sirho))$-module with
 character $\eta'$ such that
 \[ \eta'(\T_{\si(w),\sirho})= \delta_{\la,\si}^{-1}(w)\,\eta(\T_{w,\rho})
   \quad\text{ for all }w\in W(\la).\]
 The $\CC G$-modules $e_{\eta,\rho}\,\fF(\rho)$ and
 $(e_{\eta',\sirho}\,\fF_{\sirho})^\si$ are isomorphic.
\end{prop}

\begin{proof}
Since $\wi$ is an isomorphism of algebras we see from
Proposition~\ref{prop:wi(T)} that if $R(\sila)\leq \ker(\delta_{\la,\si})$
then
$$\wi(e_{\eta,\rho})=\frac{1}{c_{\eta,\rho}}\sum_{w\in W(\la)}
  \eta(\T_{w,\rho})\,\delta_{\la,\si}(\si(w^{-1}))\,\T_{\si(w),\sirho}^\vee$$
is a central-primitive idempotent of $\End_{\CC G}(\fF(\rho^{\si}))$.
Let $\eta'$ be the character of the associated simple
$\End_{\CC G}(\fF(\sirho))$-module. Then comparison of coefficients between
$\wi(e_{\eta,\rho})$ and
$$e_{\eta',\sirho}=\frac{1}{c_{\eta',\sirho}}
   \sum_{w\in W(\sila)} \eta'(\T_{w,\sirho})\,\T_{w,\sirho}^\vee$$
at $w=1$ gives
$$\frac{\eta(\T_{1,\rho})}{c_{\eta,\rho}}
  =\frac{\eta'(\T_{1,\sirho})}{c_{\eta',\sirho}}.$$
Since $\T_{1,\rho}$ is the identity element of $\End_{\CC G}(\fF(\rho))$, and
$\eta,\eta'$ have the same degree, this implies
$c_{\eta,\rho}=c_{\eta',\sirho}$. Then comparison of coefficients at
arbitrary $w\in W(\la)$ gives the first statement. The second is also clear
as $\wi$ is a $\CC G$-module isomorphism.
\end{proof}

\subsection {The generic algebra $\cH$}   \label{sec:labelling}
We next analyse in more detail the bijection between
$\Irr(\End_{\CC G}(\fF(\rho)))$ and  $\Irr(W(\la))$ using the approach
presented in \cite[Sec.~4]{HL83}. The main idea is to introduce a generic
algebra over a polynomial ring $\CC[ u_\al\mid \al\in \Delta_\la]$. One
specialisation then gives the endomorphism algebra
$\End_{\CC G}(\fF(\rho))$ and another specialisation gives the group algebra
$\CC W(\la)$. Application of these specialisations to the irreducible
characters defines a parametrisation of the constituents of $\R_L^G(\la)$ by
$\Irr(W(\la))$.

Let $\bu=(u_\al \mid \alpha \in \Delta_\la)$ be indeterminates with
$u_\al=u_\beta$ if and only if $\al$ and $\beta$ are conjugate under $W(\la)$.
Let $K$ be an algebraic closure of the quotient field of the Laurent series
ring $A_0=\CC[\bu^{\pm1}]$, and let $A$ be the integral closure of $A_0$
in $K$. Let $\cH$ be the free $A$-module with basis $\{a_w\mid w\in W(\la)\}$.
According to \cite[4.1]{HL83} one can define a unique $A$-bilinear associative
multiplication on $\cH$ such that for all $x\in C(\la)$, $w\in W(\la)$ and
$\al\in \Delta_\la$ one has
\begin{eqnarray*}
a_xa_w&=&a_{xw} \text{ and } a_wa_x=a_{wx},\\
a_{s_\al}a_w&=&\begin{cases} a_{s_\alpha w}&
    \text{ if } w^{-1} \alpha \in \Phi_\la^+,\\
 u_\alpha a_{s_\alpha w} +(u_\alpha -1) a_w&
    \text{ if } w^{-1} \alpha \notin \Phi_\la^+, \end{cases}\\
a_wa_{s_\al}&=&\begin{cases} a_{ws_\al} &\text{ if } w \alpha \in \Phi_\la^+,\\
                             u_\alpha a_{ws_\al} +(u_\alpha -1) a_w&
    \text{ if } w \alpha \notin \Phi_\la^+.\end{cases}
\end{eqnarray*}

Any homomorphism $f: A \rightarrow \CC$ induces a right $A$-module
structure on the field $\CC$, so we obtain from $\cH$ a $\CC$-algebra
$\cH^f:=\CC \otimes_A \cH$ with $\CC$-vector space basis
$\Set{1 \otimes a_w |w \in W(\la)}$. The structure constants of $\cH^f$
are obtained from the ones of $\cH$ by applying~$f$.

By \cite[4.2]{HL83} the morphisms $f_0,g_0: A_0 \rightarrow \CC$
defined by $f_0(u_\al)=p_{\al,\la}$ and $g_0(u_\al)=1$ for
$\alpha \in \Delta_\la$ can be extended to morphisms
$f,g: A \rightarrow \CC$. Then $\cH^f$ is isomorphic to
$\End_{\CC G}(\fF(\rho))$ via $1\otimes a_w\mapsto \T_{w,\rho}$ and
$\cH^g$ is isomorphic to $\CC W(\la)$ via $1\otimes a_w\mapsto w$.

By \cite[4.7]{HL83} the map $\eta\mapsto \eta^f$ with
$\eta^{f}(1\otimes a_w):=f(\eta(a_w))$ defines a bijection between the
set of $K$-characters associated to simple $K\otimes_A \cH$-modules
and the characters associated to simple $\cH^f$-modules. The analogous result
holds for $\cH^g$. This combines to give a bijection between
$\Irr(W(\la))$ and $\Irr(\End_{\CC G}(\fF(\rho)))$ and thus provides a
labelling of the irreducible
constituents of $\R_L^G(\la)$ by $\Irr(W(\la))$: for $\eta\in\Irr(W(\la))$
we denote by $\R_L^G(\la)_{\eta}$ the irreducible character of $G$ occurring
in $e_{{\eta',\rho}^f} \fF(\rho)$, where $\eta'$ is the $K$-character of $\cH$
with $\eta'^{g}=\eta$.

Together with Proposition~\ref{prop:wi(e)} this proves:

\begin{thm}   \label{thm:equiv_HC}
 If $R(\sila)\leq \ker(\delta_{\la,\sigma})$, then for $\eta\in \Irr(W(\la))$
 we have
 $$^\si(\R_L^G(\la)_\eta)=\R_L^G(\sila)_{\eta'}
 \label{equiv_equation}$$
 with $\eta':=\sieta\delta_{\la,\si}^{-1}$.
\end{thm}

\subsection {Uniqueness of parametrisation}   \label{subsec:unique}
So far, our parametrisation of constituents of $\R_L^G(\la)$ and hence also
the assertion of Theorem~\ref{thm:equiv_HC} depends on the choice of the
parabolic subgroup $P$ containing $L$. The following result, which extends
\cite[Thm.~2.12]{McGovern} from the case of a torus to an arbitrary Levi
subgroup, allows us to control that dependency.

\begin{thm}   \label{thm:action N} \label{equiv_HC}
 Let $n\in N(L)$. Assume that the parametrisation of the constituents of
 $\R_L^G(\la)$ and $\R_L^G(^n\la)$ is obtained using the same parabolic
 subgroup $P$ of $G$ with $L\le P$ and extensions of $\la$ and $^n\la$
 given by an $N(L)$-equivariant extension map. Then
 $$\R_L^G(\la)_\eta=\R_L^G(^n\la)_{^n\eta},$$
 where $^n\eta\in\Irr(W(^n\la))$ is the character with $^n\eta(^nx)=\eta(x)$
 for $x\in W(\la)$.
\end{thm}

\begin{proof}
Write $w$ for the image of $n$ in $W$. Note that by multiplying $n$ by elements
of $L$ we may assume that $w$ fixes $\Delta_L$, and also that $w$ preserves
the set of positive roots $\Phi_\la^+$ (by multiplying with a suitable element
from $R(\la)$). By \cite[10.1.3]{Ca85}, for $v\in W$ the map
$$\theta_v:\fF(\rho)\rightarrow\fF(^v\!\rho),\quad
  \theta_v(f)(x):=f(\dot{v}e_Ux)\quad\text{for $f\in\fF(\rho)$, $x\in G$},$$
is a homomorphism of $G$-modules. Moreover, it is invertible by
\cite[10.5.1, 10.5.3]{Ca85}. Now let $v\in W(\la)$ and set $v':=wvw^{-1}$. It
then follows by \cite[10.7.5]{Ca85} that
$$\theta_w\theta_v=\sqrt{\frac{\ind(wv)}{\ind(w)\ind(v)}}\theta_{wv}\quad
  \text{and}\quad
  \theta_{v'}\theta_w=\sqrt{\frac{\ind(v'w)}{\ind(w)\ind(v')}}\theta_{v'w},$$
so that
$$\theta_w\theta_v\theta_w^{-1}=\sqrt{\frac{\ind(v')}{\ind(v)}}\theta_{v'}.$$
Now we have that $\rB_{v,\rho}=\w\rho(\dot{v})\circ\theta_v$, and that
$\theta_w\circ\w\rho(\dot{v})=\w\rho(\dot{v})\circ\theta_w$ by the argument
given in the proof \cite[Prop.~10.2.4]{Ca85}. Thus
$$\begin{aligned}
  \theta_w\circ\rB_{v,\rho}\circ\theta_w^{-1}
  =&\theta_w\circ\w\rho(\dot{v})\circ\theta_v\circ\theta_w^{-1}
  =\w\rho(\dot{v})\circ\theta_w\circ\theta_v\circ\theta_w^{-1}\\
  =&\w\rho(\dot{v})\circ\sqrt{\frac{\ind(v')}{\ind(v)}}\theta_{v'}
  =\sqrt{\frac{\ind(v')}{\ind(v)}}{}^n\!\w\rho(\dot{v'})\circ\theta_{v'}
  =\sqrt{\frac{\ind(v')}{\ind(v)}}\rB_{v',{}^n\!\rho}.
\end{aligned}$$
Comparing the quadratic polynomials satisfied by $\rB_{s_\al,\rho}$ and
$\rB_{{^w}s_\al,{}^n\!\rho}$ we see that $\eps_{\al,\la}$ and
$\eps_{w(\al),{}^w\!\la}$ agree for $\al\in\Delta_\la$. Also, as conjugation by
$n$ does not change the degrees of the two constituents of $\R_L^{L_\al}(\la)$
we have $p_{\al,\la}=p_{w(\al),{}^w\!\la}$. Thus, the isomorphism
of $G$-modules $\theta_w$ sends the standard generators $\T_{v,\rho}$ of
$\End_{\CC G}(\fF(\rho))$ to the generators $\T_{{^n}v,{}^n\!\rho}$ of
$\End_{\CC G}(\fF(^n\rho))$. It then follows from our construction of the
central primitive idempotents and the specialisation argument as in the proof
of Theorem~\ref{thm:equiv_HC} that conjugation by $n$ sends the character of
$\End_{\CC G}(\fF(\rho))$ parametrised by $\eta\in\Irr(W(\la))$ to the
character parametrised by $^n\eta\in\Irr(W(^n\la))$.
\end{proof}

\section{The stabilisers of some Harish-Chandra induced Characters}
\noindent
Using Proposition~\ref{prop:5_11_here} for characters lying in the principal
Harish-Chandra series and Proposition~\ref{prop:loc_param_C} together with
the results from Section \ref{sec:HC} we can show that
Harish-Chandra induction induces an equivariant map between certain local
characters and suitable characters of $G=\GF$. For this we determine the
stabilisers of characters $\chi\in\Irr(G)$ lying in a Harish-Chandra series
$\cE(G,(L,\la))$ where $L$ is either
\begin{itemize}
 \item a maximally split torus, or
 \item the standard Levi subgroup of type $\tC_1$ in $\bG$ of type $\tC_l$
  from Section~\ref{type C}.
\end{itemize}
In order to assure the assumptions made in Section \ref{sec:HC} we first
describe the groups $R(\la)$. Recall that $\bB^F$ and $\NNN_\bG(\bT)^F$ form
a split $BN$-pair in $G$, see \cite[Thm.~24.10]{MT} for example.
Let $\Phi_1$ be the associated root system of $G$ and $X_\al\leq G$
($\al\in \Phi_1$) the associated root subgroups.
In the following we use freely the notation around Harish-Chandra
induction introduced in the section before.

\begin{lem}   \label{lem:Rla}
 Assume that $\bG$ is not of type $\tA_l$. Let $\la\in\Irr(\bT^F)$,
 $\al\in\Phi_1$ and let $\Phi_\la$ be defined as in \ref{subsec:p alpha}.
 Then $\al\in \Phi_\la$ if and only if
 $\la(\bT^F\cap \spann<X_{\al},X_{-\al}>)=1$.
 Moreover $R(\la)\leq W(\w\la)$ for any $\w\la\in\Irr(\w\bT^F|\la)$.
\end{lem}

\begin{proof}
Assume first that $F$ is a standard Frobenius endomorphism. Since $T=\bT^{F}$
is a torus the integer $p_{\al,\la}$ from Section~\ref{subsec:p alpha} is
determined inside the standard Levi subgroup
$L_\al:=\langle T,X_{\pm\al}\rangle$. As $\bG$ is of simply connected type,
the group $K:=\langle X_{\pm\al}\rangle$ is isomorphic to $\SL_2(q)$. Moreover
$T_0:=\bT\cap\spann<X_{\pm\al}>=\spann<h_\al(t)\mid t\in \FF_q^\times>$.
Let $\la_0:=\restr \la| {T_0}$.

From the situation in $\SL_2(q)$ we know that $\R_{T_0}^K(\la_0)$ splits into
two constituents of different degrees if $\la_0$ is trivial. Hence in that
case $p_{\al,\la}\neq 1$. If $\la_0$ is not trivial then either the character
$\R_{T_0}^K(\la_0)$ and hence $\R_T^{L_\al}(\la)$ is irreducible or it is the
sum of two characters of the same degree.

The above argument remains valid for twisted Frobenius endomorphisms,
possibly replacing $\SL_2(q)$ by $\SL_2(q^m)$ with $m$ the order of the
graph automorphism induced by~$F$.

Let $\w\la\in\Irr(\w\bT^F|\la)$. If $\al\in \Phi_\la$ and
$s_\al\in \pi(\NNN_G(\bT))$ is the reflection associated with $\al$
the Steinberg relations imply
$$[\w\bT^{F},s_\al]\subseteq \bT^F\cap \spann<X_{\al},X_{-\al}>.$$
Together with the above we conclude that $s_\al\in W(\wla)$ if
$\al\in\Phi_\la$. Since $R(\la)$ is generated by the elements $s_\al$
($\al\in\Phi_\la$) this proves the statement.
\end{proof}

\begin{lem}   \label{lem:C(la)=1}
 If $\Phi$ is of type $\tC_l$ and $L$ is a standard Levi subgroup of type
 $\tC_1$ then for every $\la\in\Irr(L)$ and $\wla\in\Irr(\w L|\la)$ we have
 $R(\la)\leq W(\wla)$.
\end{lem}

\begin{proof}
Let $\la\in\Irr(L)$, so $\la=\la_1\times\zeta$ with $\la_1\in\Irr(T_1)$ and
$\zeta\in\Irr(L_0)$ where $T_1$ and $L_0$ are as defined in
Section~\ref{type C}.
For $N$ defined as there $W(\la)=N_\la/T\cong W(\la_1)$, where $W(\la_1)$ is
the relative Weyl group of $\la_1$ in the subgroup of type $\tC_{l-1,\SC}(q)$
centralising $L_0$. Let $\wla\in\Irr(\w L|\la)$. Then by the direct
product structure of $N$ we see that $W(\wla)=W(\wla_1)$ for some
$\wla_1\in \Irr(\w T_1|\la_1)$. Thus we are in the situation described in
Lemma~\ref{lem:Rla}, but with respect to the torus $T_1$ in type $\tC_{l-1}$,
whence $W(\wla)=W(\wla_1)\geq R(\la_1)$.

Let $\{\pm\al_1\}$ denote the root system of the standard Levi subgroup $L$.
Let $\Omega\subseteq \Phi\setminus\{\pm\al_1\}$ be defined as in
\ref{subsec:p alpha} and $\al\in\Omega$. Computations in $W$ show that $\al$
is orthogonal to $\al_1$. Let $L_\al$ be the standard Levi subgroup of $G$
corresponding to the simple system $\{\al_1,\al\}$.
Then $\R_L^{L_\al}(\la)=\R_{T_1}^{\spann<T_1,X_{\pm \al}>}(\la_1)\times \zeta$.
Accordingly $\al\in \Phi_\la$ if and only if $\al\in \Phi_{\la_1}$, where
$\Phi_{\la_1}$ is associated with $\la_1$ as in~\ref{subsec:p alpha}. This
proves the statement, since $R(\la)=\spann<s_\al\mid\al\in\Phi_\la>$ and
$R(\la_1)=\spann<s_\al\mid\al\in\Phi_{\la_1}>$.
\end{proof}

\begin{thm}   \label{thm:bij}
 Let $L$ be either a maximally split torus of $(\bG,F)$ or the Levi subgroup
 from Section~\ref{type C} in type $\tC_l$. Let $N\deq \NNN_G(\bS)$ and
 $\w N\deq \NNN_\wG(\bS)$ where $\bS$
 is a torus of $\bG$ such that $\Cent_{G}(\bS)=L$.
 Then there exists an $\w N D$-equivariant bijection
 $$\Irr_\cu(N)\longrightarrow  \bigcup_{\la\in \Irr_\cu(L)}\cE(G,(L,\la)),$$
 where $\Irr_\cu(L)$ is the set of cuspidal characters of $L$,
 $\Irr_\cu(N)\deq \Irr(N|\Irr_\cu(L))$ and 
 $\cE(G,(L,\la))$ denotes the set of constituents of $\R_L^G(\la)$.
\end{thm}

\begin{proof}
In Propositions~\ref{prop:5_11_here} and~\ref{prop:loc_param_C} we gave a
parametrisation of the set $\Irr_\cu(N)$.

Let us first assume that $L$ is a maximally split torus and hence $N=\bN^F$.
In this case  $\Irr_\cu(L)=\Irr(L)$ and hence $\Irr_\cu(N)=\Irr(N)$.
Let $\Lambda$ be the
$ND$-equivariant extension map from Corollary~\ref{cor:ker_delta_sc} applied
with $d=1$ and $v=1$. Then Proposition~\ref{prop:5_11_here} yields a map
\[\Pi:\cP\longrightarrow\Irr(N),\qquad
  (\la,\eta)\longmapsto (\Lambda(\la)\eta)^{\bN^{F}},\]
with $\cP=\{(\la,\eta)\mid\la\in\Irr(L),\,\eta\in\Irr(W(\la))\}$.
On the other hand let
\[\Pi':\cP\longrightarrow \bigcup_{\la\in \Irr_\cu(L)} \cE(G,(L,\la)),
  \qquad (\la,\eta)\longmapsto \R^G_L(\la)_\eta,\]
where $\R^G_L(\la)_\eta $ is defined using $\Lambda$.
The maps $\Pi$ and $\Pi'$ induce bijections between the set of $N$-orbits in
$\cP$ and the characters in $\Irr_\cu(N)$ and
$\bigcup_{\la\in\Irr_\cu(L)} \cE(G,(L,\la))$ respectively, see
Proposition~\ref{prop:5_11_here}(1) for $\Pi$ and Theorem~\ref{thm:action N}
for the statement about $\Pi'$.
Accordingly the concatenation $\Pi'\circ\Pi^{-1}$ gives the required bijection
\[ \Irr_\cu(N) \rightarrow \bigcup_{\la\in \Irr_\cu(L)} \cE(G,(L,\la)),
  \qquad \Pi(\la,\eta) \mapsto \Pi'(\la,\eta).\]

Now the action of $\w N D$ on $\Irr(N)$ has been described in
Proposition~\ref{prop:5_11_here} in terms of the associated labels.
Analogously Theorem~\ref{thm:equiv_HC} determines the action of $\w T D$
on the sets $\cE(G,(L,\la))$ in
terms of the associated labels. (Note that by Proposition~\ref{prop:3_12} the
map $\Lambda$ satisfies the requirements made in Theorem~\ref{thm:equiv_HC}.)
Comparing the induced actions on $\cP$
we see that the bijection is $\w N D$-equivariant.

Similarly, when $L$ is as in Section~\ref{type C}, then according to the
results of Propositions~\ref{prop:ext_map_C} and \ref{prop:loc_param_C} the
above construction again determines an equivariant bijection, using that the
assumption of Theorem~\ref{thm:equiv_HC} is satisfied by
Lemma~\ref{lem:C(la)=1}.
\end{proof}

\begin{cor}   \label{cor:7_3}
 Let $(L,\la)$ be a cuspidal pair as in Theorem~\ref{thm:bij}. Then for every
 character $\chi_0\in\cE(G,(L,\la))$ there exists some $\w G$-conjugate $\chi$
 such that
 \[ (\w G D)_\chi=\w G_\chi D_\chi.\]
\end{cor}

\begin{proof}
Via the bijection from Theorem~\ref{thm:bij} the character $\chi_0$ corresponds
to a character $\psi_0\in\Irr_\cu(N)$. Some $\w N$-conjugate $\psi$ of $\psi_0$
satisfies $(\w N D)_\psi=\w N_\psi(ND)_\psi$, see Theorems~\ref{thm:IrrN_autom}
and~\ref{thm:stab_C} together with Lemma~\ref{lem:3_21}. Since the bijection
from Theorem~\ref{thm:bij} is $\w N D$-equivariant, this implies
\begin{equation*}
  (\w G D)_\chi=  G(\w N D)_\chi= G(\w N D)_\psi =G (\w N_\psi (ND)_\psi)
  = \w G_\chi (N D_\chi) = \w G_\chi D_\chi
\end{equation*}
where $\chi$ corresponds to $\psi$ via Theorem~\ref{thm:bij},
so is $\w G$-conjugate to $\chi_0$.
\end{proof}

\section{Towards the inductive McKay condition}   \label{sec:indMcK}
\noindent
In this section we collect the previous results to prove
Theorem~\ref{thm:d=1good} on the inductive McKay condition for primes $\ell$
with $d_\ell(q)=1$, whenever $\GF\notin\{\tD_{l,\SC}(q),\tE_{6,\SC}(q)\}$.
Here $d_\ell(q)$ denotes the order of $q$ modulo~$\ell$ if $\ell>2$,
respectively the order of $q$ modulo~4 if $\ell=2$. We describe under which
additional assumptions this result can be extended to primes with
$d_\ell(q)=2$ and to the missing types.

We first collect some cases in which the assertion of
Theorem~\ref{thm:d=1good} had already been
proven.

\begin{prop}  \label{prop:exc}
 Assume that $S:=\bG^F/\Z(\bG^F)$ is simple, and let $\ell$ be a prime dividing
 $|S|$. The inductive McKay condition holds for $S$ and $\ell$
 if one of the following is satisfied:
 \begin{itemize}
  \item $\ell=p$,
  \item $\Z(\bG^F)=1$,
  \item $\Phi$ is of type $\tA_l$, or
  \item $\ell=2$ and $\bG^F=\tC_{l,\SC}(q)$, where $q$ is an odd power of
   an odd prime.
 \end{itemize}
\end{prop}

\begin{proof}
The case where $\ell$ is the defining characteristic has been settled in
\cite[Thm.~1.1]{Sp12}. The case $\Z(\bG^F)=1$ has been considered in
\cite[Thm.~A, Prop.~5.2]{CS13}. If $\Phi$ is of type $\tA_l$ the statement
follows from \cite[Thm.~A]{CS15}. According to \cite[Thm.~4.11]{MaExt} the
inductive McKay condition is satisfied for $S$ and $\ell=2$ whenever
$\bG^F=\tC_{l,\SC}(q)$ for an odd power $q$ of an odd prime.
\end{proof}

\begin{prop}   \label{prop:Malle}
 Let $(\bG,F)$ be as in Section~\ref{sec:Not} and $\ell$ a prime different
 from the defining characteristic with $d=d_\ell(q)\in\{1,2\}$. Assume that
 $\bG^F$ and $\ell$ are not as in Proposition~\ref{prop:exc}.
 Let $\bS$ be a Sylow $d$-torus of $(\bG,F)$. Then
 Assumption~\ref{thm2_2gen} is satisfied for $G:=\GF$, $\w G:=\w\bG^F$, $D$,
 $N:=\NNN_\GF(\bS)$ and some Sylow $\ell$-subgroup $Q$ of $G$ such that
 $\NNN_{\w\bG^F}(\bS)=\NNN_{\w\bG^F}(Q)N$.
\end{prop}

\begin{proof}
We can argue as in the proof of Lemma 7.1 of \cite{CS15}.
According to \cite[Thms.~5.14 and 5.19]{MaH0} since $\GF$ and $\ell$ are
not as in Proposition~\ref{prop:exc} there exists some Sylow
$\ell$-subgroup $Q$ of $\GF$  with $\NNN_\GF(Q)\leq \NNN_\GF(\bS)\lneq \GF$.
Since all Sylow $d$-tori of $(\bG,F)$ are $\GF$-conjugate, we can conclude
that $\NNN_\GF(\bS)$ is $\Aut(\GF)_{Q}$-stable, see also \cite[Sect.~2.5]{CS13}.

Maximal extendibility for $\GF\lhd \w \bG^F$ as required in
\ref{hauptprop_maxext_glob} was shown by Lusztig, see \cite[Prop.~10]{LuDis}
or \cite[Thm.~15.11]{CE04}. The maximal extendibility with respect to
$N\lhd \w N$ as required in \ref{hauptprop_maxext_loc} has been proven in
Corollary~\ref{cor:maxextNwN}.
\end{proof}

In our next step we establish the existence of a bijection as required in \ref{thm2_2bij}.

\begin{thm}   \label{thm:Bij_wG}
 Let $\ell$ be a prime such that $d=d_\ell(q)\in\{1,2\}$.
 Let $\bS$ be a Sylow $d$-torus of $(\bG,F)$, $N\deq \NNN_G(\bS)$ and
 $\w N\deq \NNN_\wG(\bS)$. Let $\cG\deq \Irr\big (\w G\mid \Irrl(G)\big)$
 and $\cN\deq \Irr\big (\w N\mid \Irrl(N)\big )$. Then there is a
 $(\wG\rtimes D)_{\bS}$-equivariant bijection
 \[\w \Omega: \cG \longrightarrow \cN\]
 with $\w\Omega(\cG\cap\Irr(\w G\mid \nu))= \cN\cap\Irr(\w N\mid \nu)$ for
 every $\nu\in\Irr(\Z(\w G))$, and
 $\w\Omega(\chi\delta)=\w\Omega(\chi)\restr\delta|{\w N}$
 for every $\delta \in \Irr(\w G\mid 1_G)$ and $\chi\in\cG$.
\end{thm}

\begin{proof}
According to Corollary~\ref{cor:ker_delta_sc} there exists an
$\NNN_{\w G D}(\bS)$-equivariant extension map $\Lambda$ with respect to
$\Cent_{\w G}(\bS)\lhd \w N$ that is compatible with multiplication
by linear characters of $\w G$. The considerations made in Section 6 of
\cite{CS15} for groups of type $\tA_l$ apply in our more general situation
as well. Using our map $\Lambda$ the construction presented there gives the
required bijection.
\end{proof}

\begin{thm}   \label{thm:8_1}
 Let $(\bG,F)$ be as in Section~\ref{sec:Not} and $\ell$ a prime different
 from the defining characteristic of $\bG$ with $d=d_\ell(q)\in\{1,2\}$. Assume
 that $\bG^F$ and $\ell$ are not as in Proposition~\ref{prop:exc}, and that
 $\bG^F$ is the universal covering group of $S=\bG^F/\Z(\bG^F)$. Then the
 inductive McKay condition holds for $S$ and $\ell$ in any of the following
 cases:
 \begin{enumerate}[label=\rm(\alph*),ref=(\alph{enumi})]
  \item $(\bG,F)$ is of type $\tB_l$, $\tC_l$, $\tw 2 \tD_l$,
   $\tw 2\tE_6$ or $\tE_7$ and $d=1$;
   \item $(\bG,F)$ is of type $\tD_l$ or $\tE_6$, $d=1$,
   and~\ref{2_2gloext} holds;
  \item $(\bG,F)$ is of type $\tB_l$, $\tC_l$, $\tw2\tD_l$, $\tw2\tE_6$
   or $\tE_7$, $d=2$ and~\ref{2_2glostar} holds; or
  \item $(\bG,F)$ is of type $\tD_l$ or $\tE_6$, $d=2$, and~\ref{2_2glostar}
   and~\ref{2_2gloext} hold.
 \end{enumerate}
\end{thm}

\begin{proof}
This is proven by an application of Theorem~\ref{thm:Sp12}. We
successively ensure that the necessary assumptions are satisfied.

The groups $G:=\GF$, $\w G:=\w\bG^F$, $D$, $N$ and $Q$ are chosen as in
Proposition~\ref{prop:Malle} and satisfy accordingly the assumptions made
in~\ref{thm2_2gen}. For this group $N$ the characters satisfy
assumption~\ref{thm2_2loc} according to Theorem~\ref{thm:IrrN_autom}.

Let $\chi\in\Irr(\GF)$ lie in a Harish-Chandra series
$\cE(\GF,(\bT^F,\la))$ for some character $\la\in\Irr(\bT^F)$. Then
\ref{2_2glostar} holds for $\chi$ (after suitable $\wGF$-conjugation)
according to Corollary~\ref{cor:7_3}.

Now assume that $d=1$. Then according to \cite[Prop.~7.3]{MaH0} each
character in $\Irrl(\GF)$ lies in a Harish-Chandra series
$\cE(\GF,(\bT^F,\la))$ for some character $\la\in\Irr(\bT^F)$. So
assumption~\ref{2_2glostar} holds again by Corollary~\ref{cor:7_3}. On the
other hand \ref{2_2gloext} clearly holds whenever $D$ is cyclic or by
assumption. If $(\bG,F)$ of type $\tB_l$, $\tC_l$, $\tw 2 \tD_l$, $\tw 2\tE_6$
or $\tE_7$ then $D$ is cyclic.

Whenever $d\in\{1,2\}$ the bijection from Theorem~\ref{thm:Bij_wG} has the
properties required in \ref{thm2_2bij}.
Altogether this proves the above statements.
\end{proof}

\begin{rem}
Note that the equation given in \ref{2_2glostar} has only to be checked for
$\ell'$-characters of $G$ that are not $\w G$-invariant since every
$\chi\in\Irr(G)$ with $\w G_\chi=\w G$ satisfies
$(\w G \rtimes D)_\chi=\w G\rtimes D_\chi$. In particular only characters in
Lusztig rational series $\cE(G,s)$ have to be considered where the centraliser
of $s$ in the dual group $\bG^*$ is not connected, since characters in Lusztig
series corresponding to elements with connected centralisers are
$\w G$-invariant according to \cite[Prop.~5.1]{LuDis}, see also
\cite[Cor.~15.14]{CE04}.

Furthermore \ref{2_2glo} holds whenever $\Aut(S)/S$ is cyclic since then
$(\wG\rtimes D)_\chi/(G\Z(\wG))$ is cyclic and hence coincides with
$(\wG_\chi \rtimes D_\chi)/(G\Z(\wG))$.
\end{rem}

Theorem~\ref{thm:d=1good} in the case where $\bG^F$ is a universal covering
group and $d_\ell(q)=1$ is now part~(a) of the preceding theorem. For odd
primes $\ell$ the following completes the proof of Theorem~\ref{thm:d=1good}.
The cases where the universal covering group of a simple group $S$ is
not of the form $\bG^F$ can be determined by Table 6.1.4 of \cite{GLS3}.
Then the Schur multiplier of $S$ is said to have a non-trivial exceptional
part, see \cite[Sec.~6.1]{GLS3}.

\begin{lem}   \label{lem:excSchur}
 Let $S$ be a simple group of Lie type with a non-trivial exceptional part of the Schur multiplier
 and let $\ell$ be  a prime dividing $|S|$. Assume that $\ell$ is the defining
 characteristic or $d_\ell(q)\in \{1,2\}$. Then the inductive McKay condition
 holds for $S$ and $\ell$.
\end{lem}

\begin{proof}
If $S$ is a Suzuki or Ree group the result is known from \cite[Thm.~16.1]{IMN}
and \cite[Thm.~A]{CS13}. Otherwise $S\cong\bG^F/Z(\bG^F)$ for some pair
$(\bG,F)$ as in Section \ref{sec:Not}. If $\ell$ is the defining characteristic
of $\bG$ the statement follows from \cite[Thm.~1.11]{Sp12}. If $(\bG,F)$ is of
type $\tA_l$, $\tw 2\tA_l$, $\tF_4$, or $\tG_2$ then the claim is known by
\cite[Thm.~A]{CS13} and \cite[Thm.~A]{CS15}.

In the other cases the considerations from \cite[Sec.~7]{CS15} can be
transferred: the inductive McKay condition holds for $S$ if it holds for any
pair $(S,Z)$ in the sense of \cite[Def.~7.3]{CS15}, where $Z$ is a cyclic
$\ell'$-quotient of the Schur multiplier of $S$, see \cite[Lemma~7.4(a)]{CS15}.
Taking into account \cite[Thm.~1.1]{ManonLie} it is sufficient to prove the
claim in the cases where $Z$ is a quotient of the non-exceptional Schur
multiplier of $S$.

According to \cite[Table~6.1.3]{GLS3} all Sylow subgroups of $D$ are cyclic
and every automorphism of $\bG^F$ is induced by $\w \bG^F D$ for
groups $\w\bG^F$ and $D$ defined as in Section~\ref{sec:Not}.
Further in those cases any Sylow subgroup of the outer automorphism group
of $S$ is cyclic and hence \ref{2_2glo} holds.
The proofs of Theorem \ref{thm:8_1} and \cite[Thm.~2.12]{Sp12} imply
that the inductive McKay condition holds for $(S,Z)$ in those missing cases
if $d=d_\ell(q)\in \{1,2\}$.
\end{proof}

\section{The McKay conjecture for $\ell=2$}   \label{sec:l=2}
\noindent
In this section we prove Theorem~\ref{thm:McKayp=2} from the introduction.
Let $\bG$ be a connected reductive linear algebraic group over an algebraically
closed field of characteristic~$p$ and $F:\bG\rightarrow\bG$ a Steinberg
endomorphism defining an $\FF_q$-structure on $\bG$ such that $G:=\bG^F$ has
no component of Suzuki or Ree type.

\subsection {Degree polynomials}   \label{subsec:degpol}
We begin by defining degree polynomials for the irreducible characters
of $G=\bG^F$, which play a crucial role in our arguments. These are probably
known to (some) experts, but we have not been able to find a convenient
reference. Let $\GG$ be the complete root datum of $(\bG,F)$. Then there is
a monic integral polynomial $|\GG|\in\ZZ[X]$ such that $|\bG^{F^m}|=|\GG|(q^m)$
for all natural numbers $m$ prime to the order of the automorphism induced
by $F$ on the Weyl group $W$ of $\bG$ (see \cite[1C]{BM92}).
The same then also holds for any connected reductive $F$-stable
subgroup of $\bG$, like maximal tori or connected components of centralisers.
Furthermore, by work of Lusztig the unipotent characters of any finite
reductive group with complete root datum $\GG$ are parameterised uniformly,
and the degree of a unipotent character $\chi$ of $\bG^F$ is given by
$f_\chi(q)$ for a suitable polynomial $f_\chi\in\QQ[X]$ depending only on the
parameter of $\chi$, see \cite[\S1B]{BMM}.
Now let $\chi\in\Irr(G)$ be arbitrary. Then $\chi$ lies in the Lusztig series
$\cE(G,s)$ of a semisimple element $s$ of the dual group $G^*=\bG^{*F}$, and
Lusztig's Jordan decomposition of characters gives a bijection
\begin{align}
  \Psi:\cE(G,s)\rightarrow \cE(\Cent_{G^*}(s),1)\quad\text{ such that }\quad
  \chi(1)=|G^*\co \Cent_{G^*}(s)|_{p'}\,\Psi(\chi)(1),\label{eq:Jordan dec}
\end{align}
where $\cE(\Cent_{G^*}(s),1)$ denotes the unipotent characters of
$\Cent_{G^*}(s)$, see \cite[Thm.~13.25]{DM}. While this bijection is not
defined canonically, the formula in loc.~cit.\ for scalar products with
Deligne--Lusztig characters shows that its uniform projection is, and hence
in particular so is the correspondence of degrees. Moreover, by the description
in \cite[Prop.~5.1]{LuDis}, the multiplicities of unipotent characters of
$\Cent_{G^*}^\circ(s)$ in those of $\Cent_{G^*}(s)$ are determined by the
complete root datum of $(\bG,F)$, hence generic, so there exist well-defined
degree polynomials $f_\psi$ for the unipotent characters $\psi$ of the
possibly disconnected group $\Cent_{G^*}(s)$. Thus, denoting by $|\GG_s|$
the order polynomial of $\Cent_{G^*}(s)$, we can define
from~\eqref{eq:Jordan dec} the \emph{degree
polynomial} $f_\chi:=(|\GG|/|\GG_s|)_{X'}f_{\Psi(\chi)}\in\QQ[X]$ of $\chi$.

\begin{lem}   \label{lem:degHCseries}
 Let $G=\bG^F$ be as above. If $\chi\in\Irr(G)$ lies in the Harish-Chandra
 series of a cuspidal character of a Levi subgroup of $G$ of semisimple
 $\FF_q$-rank $r$ then its degree polynomial $f_\chi$ is divisible by
 $(X-1)^r$.
\end{lem}

\begin{proof}
Let $\bL\le\bG$ be an $F$-stable Levi subgroup such that $\chi$ lies in the
Harish-Chandra series $\cE(G,(L,\la))$, where $L=\bL^F$. Let $\bT\le\bG$ be an
$F$-stable maximal torus of $\bG$, with Sylow 1-torus $\bT_1\le\bT$. Then
$\bM=\Cent_\bG(\bT_1)$ is a (1-split) Levi subgroup of $\bG$, and
$$\langle \R_T^G(\theta),\chi\rangle=\langle \R_M^G(\mu),\chi\rangle$$
with $\mu=\R_T^M(\theta)$ a (virtual) character of $M=\bM^F$, where $T=\bT^F$.
Thus, if $\langle \R_T^G(\theta),\chi\rangle\ne0$ then by disjointness of
Harish-Chandra series we must have $\bM\ge \bL$ up to conjugation, whence
$\bT_1\le \Z(\bM)_{\Phi_1}\le \Z(\bL)_{\Phi_1}$, where $\Z(\bM)_{\Phi_1}$
and $\Z(\bL)_{\Phi_1}$ denote the Sylow $1$-torus of the groups $\Z(\bM)$
and $\Z(\bL)$.
\par
Now we have $\chi(1)=\langle \reg_G,\chi\rangle$, where the regular character
$\reg_G$ of $G$ is given by
$$\reg_G=\frac{1}{|W|}\sum_{w\in W}|G\co T_w|_{p'}\, \R_{T_w}^G(\reg_{T_w}),$$
where $W$ is the Weyl group of $\bG$, $T_w=\bT_w^F$ is a maximal torus of $G$
of type $w$, and $\reg_{T_w}$ denotes the regular character of $T_w$ (see
\cite[Cor.~12.14]{DM}). Let $W_0$ denote the set of elements $w\in W$
satisfying $\dim(\bT_w)_{\Phi_1}\le\dim (\Z(\bL)_{\Phi_1})$. Our above
considerations then yield
$$\chi(1)= \frac{1}{|W|}\sum_{w\in W_0}|G\co T_w|_{p'}\,
           \langle \R_{T_w}^G(\reg_{T_w}),\chi\rangle.$$
(Note that this is generic, as by \cite[Thm.~4.23]{Lu} the multiplicities
$\langle \R_{T_w}^G(\reg_{T_w}),\chi\rangle$ only depend on the unipotent
Jordan correspondent of $\chi$.) For any $F$-stable reductive subgroup $\bH$
of $\bG$ we write in the following $\bH_{\Phi_1}$ for a Sylow $1$-torus of
$(\bH,F)$. For $w\in W_0$ we have
$$\dim(\bG_{\Phi_1})-\dim(({\bT_w})_{\Phi_1})\ge
   \dim(\bL_{\Phi_1})-\dim (\Z(\bL)_{\Phi_1})=r,$$
where $r$ is the semisimple $\FF_q$-rank of $\bL$, so the degree polynomial
$f_\chi$ of $\chi(1)$ is divisible by $(X-1)^r$.
\end{proof}

We also recall the following facts from ordinary Harish-Chandra theory (see
e.g. \cite[Thm.~10.11.5]{Ca85}). Let $L\le G$ be a Levi subgroup with a
cuspidal character $\la\in\Irr(L)$, and let $W(\la)$ denote the relative Weyl
group of this cuspidal pair. Assume that $\chi\in\Irr(G)$ lies in the
Harish-Chandra series above $(L,\la)$. Let $\eta\in\Irr(W(\la))$ be the
character associated to $\chi$ and $D_\chi\in\QQ(X)$ the inverse of the Schur
element of $\eta$ of the corresponding generic Hecke algebra, so numerator
and denominator of $D_\chi$ are prime to $X-1$. Then
\begin{align}
  \chi(1)&=|G\co L|_{p'}\,D_\chi(q)\,\la(1)\qquad\text{and}\quad
  D_\chi(1)=\eta(1)/|W(\la)|.\label{eq:HLdeg}
\end{align}
With the degree polynomial $f_\la$ of $\la$ we define a \emph{degree function}
$f_\chi'\in\QQ(X)$ for $\chi$ as $f_\chi'=(|\GG|/|\LL|)_{X'} D_\chi f_\la$,
where $\LL$ denotes the complete root datum associated to the standard Levi
subgroup $(\bL,F)$ with $\bL^F=L$. The following is shown in
\cite[Thm.~3.2]{BMM}:

\begin{lem}   \label{lem:deguni}
 Let $G=\bG^F$ be as above. If $\chi\in\Irr(G)$ is unipotent, then
 $f_\chi=f_\chi'$.
\end{lem}


\subsection {Unipotent characters of odd degree}
From now on and for the rest of this section assume that $p$ and
hence $q$ \emph{is odd}.

\begin{prop}   \label{prop:cusp}
 Let $G=\bG^F$ be as above. Then every non-trivial cuspidal unipotent character
 $\chi$ of $G$ has even degree.
 More precisely, if $\chi$ has degree polynomial $a(X-1)^mf$, with $a\in\QQ$,
 $m\ge0$ and $f\in\ZZ[X]$ monic and prime to $X-1$, then $a(q-1)^m$ is even.
\end{prop}

\begin{proof}
First note that unipotent characters of $G$ restrict irreducibly to unipotent
characters of $[\bG,\bG]^F$, so we may assume that $\bG$ is semisimple.
Furthermore, degrees of unipotent characters are insensitive to the isogeny
type of $\bG$, whence we may assume that $\bG$ is of simply connected type
and hence a direct product of simple algebraic groups. As unipotent characters
of a direct product are the exterior products of the unipotent characters of
the factors, we may reduce to the case that $\bG$ is a direct product of $r$
isomorphic simple groups $\bH_i\cong\bH$, $1\le i\le r$, transitively permuted
by $F$. But then $\bG^F\cong\bH^{F^r}$, and $f(X^r)$ is divisible by the same
power of $X-1$ as $f(X)$, so that finally we may assume that $\bG$ is simple.
\par
We then use Lusztig's classification of cuspidal unipotent characters. In
fact, when $q\equiv1\pmod4$ then the first claim is already proved in
\cite[Prop.~6.5]{MaH0}. But a quick check of that argument shows that it
only relies on the fact that the degree of $\chi$ is divisible by a
sufficiently high power of the even number $q-1$. It thus also works for
$q\equiv3\pmod4$ and does even yield the second assertion.
\end{proof}


\begin{prop}   \label{prop:unip}
 Let $G=\bG^F$ be as above. Then all unipotent characters of $G$ of odd degree
 lie in the principal series of $G$.
\end{prop}

\begin{proof}
We distinguish two cases. If $q\equiv1\pmod4$ then our claim is contained in
\cite[Cor.~6.6]{MaH0}. So for the rest of the proof we may suppose that
$q\equiv3\pmod4$. \par
Assume if possible that $\chi$ is a unipotent character of $G$ of odd degree
and not lying
in the principal series. So $\chi$ lies above a cuspidal unipotent character
$\la\ne 1_L$ of a Levi subgroup $L\le G$. Let $f_\chi,f_\la\in\QQ[X]$ denote
the degree polynomials of $\chi,\la$ respectively. As $\chi(1)$ is odd and
$4|(q+1)$ we have that $\chi$ must lie in the principal 2-series of $G$ by
\cite[Cor.~6.6]{MaH0} applied with $d=2$. Thus $f_\chi$ is prime to $X+1$
according to \cite[Prop.~2.4]{BMM} (an analogue of our
Lemma~\ref{lem:degHCseries}).
\par
Now by what we recalled before Lemma~\ref{lem:deguni} there exists a
rational function $g\in\QQ(X)$ with numerator and denominator products of
cyclotomic polynomials times an integer, both prime to $X-1$, such that
$f_\chi= g\cdot f_\la$, and such that $g(1)$ is the degree of an irreducible
character of the relative Weyl group $W(\la)$ of $(L,\la)$ in $G$. 
Write $f_\la=(X+1)^kf_1$ with a non-negative integer $k$ such that
$f_1\in\QQ[X]$ is prime to $X+1$. Then by our observations on 
$f_\chi$ and $f_\la$ there exists a rational function $g_1$ such that
$g=g_1/(X+1)^k$ and both numerator and denominator of $g_1$ are prime
to $X^2-1$. Then $f_\chi=g_1\cdot f_1$.

Now let $\Phi_i$ be a cyclotomic polynomial dividing $g_1$. Then $\Phi_i(q)$
is odd unless $i=2^{j+1}$ for some $j\ge1$, in which case we have
$\Phi_i(q)_2=(q^{2^j}+1)_2=2=\Phi_i(1)_2$. Thus, $g_1(q)$ is divisible by
the same $2$-power as $g_1(1)$, which is an integer. Since $f_1(q)$ is even
by Proposition~\ref{prop:cusp} we conclude that
$$\chi(1)=g(q)\cdot f_\la(q)=g_1(q)\cdot f_1(q) \equiv g_1(1)\cdot f_1(q)
  \pmod 2$$
is even as well, a contradiction.
\end{proof}

\subsection {Characters of odd degree and the principal series}

\begin{lem}   \label{lem:centSyl}
 Let $\bH$ be simple of adjoint type $\tB_l$ ($l\ge1$), $\tC_l$ ($l\ge2)$,
 $\tD_{2l}$ ($l\ge2$) or $\tE_7$, and $F:\bH\rightarrow\bH$ a Steinberg
 endomorphism. Let $s\in \bH^F$ be semisimple centralising a Sylow 2-subgroup
 of $\bH^F$. Then $s^2=1$.
\end{lem}

This was observed in \cite[Lemma~4.1]{MaExt} for type $\tB_l$; the
proof given there carries over word by word, since in all listed cases the
longest element of the Weyl group acts by inversion on a maximal torus.

Recall that an element of a connected reductive algebraic group $\bH$ is called
\emph{quasi-isolated} if its centraliser is not contained in any proper Levi
subgroup of $\bH$.

\begin{lem}   \label{lem:2-central}
 Let $\bH$ be simple of adjoint type $\tB_l$ ($l\ge2$), $\tC_l$ ($l\ge3$) or
 $\tD_l$ ($l\ge4$), and $F:\bH\rightarrow\bH$ a Frobenius endomorphism defining
 an $\FF_q$-rational structure such that $q\equiv3\pmod4$. Let $s\in\bH$ be
 semisimple quasi-isolated with disconnected centraliser $\bC=\Cent_{\bH}(s)$
 such that $\bC^F$ contains a Sylow 2-subgroup of $\bH^F$. Then $\bC^F$ is
 as in Table~\ref{tab:2-central}.
\end{lem}
\begin{table}[htbp]
 \caption{Disconnected centralisers of $2$-central elements}   \label{tab:2-central}
$\begin{array}{|l|l|l|}
\hline
     \bH^F& \bC^F& \text{conditions}\cr
\hline
 \tB_l(q)& \tB_{l-2k}(q)\cdot \tD_{2k}(q).2& 1\le k\le l/2\cr
 \tB_{2l+1}(q)& \tB_{2k}(q)\cdot\tw2\tD_{2(l-k)+1}(q).2& 0\le k\le l\cr
\hline
 \tC_{2l}(q)& (\tC_l(q)\cdot \tC_l(q)).2& \cr
\hline
 \tD_l(q)& (\tD_k(q)\cdot \tD_{l-k}(q)).2& 1\le k< l/2\cr
 \tD_{4l}(q)& (\tD_{2l}(q)\cdot \tD_{2l}(q)).4& \cr
\hline
 \tw2\tD_l(q)& (\tD_k(q)\cdot \tw2\tD_{l-k}(q)).2& 2\le k\le l-1,\ k\ne l/2\cr
\hline
\end{array}$\par

 Here $\tD_1(q)$, $\tw2\tD_1(q)$ are to be interpreted as tori of order
 $q-1$, $q+1$ respectively.
\end{table}
\begin{proof}
The conjugacy classes of quasi-isolated elements $s$ in classical groups of
adjoint type were classified in \cite[Tab.~2]{Bo05}. From that list, we may
exclude those $s$ with connected centraliser. It then remains to determine the
various rational forms of $\bH$ and $\bC$ and to decide when $s$ is
2-central.
We treat the cases individually. For $\bH$ of adjoint type $\tB_l$, the table
contains all examples from loc.~cit. For $\bH$ of type $\tC_l$, an easy
calculation shows that only the listed case occurs. Similarly, it can be
checked in type $\tD_l$ from the order formulas that only the listed types of
disconnected centralisers can possibly contain a Sylow 2-subgroup.
\end{proof}

We thus obtain the following classification of characters of odd degree:

\begin{thm}   \label{thm:odd degree}
 Let $\bG$ be simple, of simply connected type, not of type $\tA_l$, with
 $F$, $\bG^*$ as introduced in Section~\ref{ssec2:B}. Let
 $\chi\in\Irr_{2'}(G)$. Then either $\chi$ lies in the principal series of $G$,
 or $q\equiv3\pmod4$, $G=\Sp_{2l}(q)$ with $l\ge1$ odd, $\chi\in\cE(G,s)$ with
 $\Cent_{G^*}(s)=\tB_{2k}(q)\cdot\tw2\tD_{l-2k}(q).2$ where $0\le k\le (l-3)/2$,
 and $\chi$ lies in the Harish-Chandra series of a cuspidal character of
 degree $\frac{1}{2}(q-1)$ of a Levi subgroup $\Sp_2(q)\times(q-1)^{l-1}$.
\end{thm}

\begin{proof}
We follow the line of arguments in \cite[\S7]{MaH0}. Let $\chi\in\Irr(G)$ be a
character of odd degree and not lying in the principal series of $G$. Then
the degree polynomial of $\chi$ is divisible by $X-1$, by
Lemma~\ref{lem:degHCseries}. Let $s\in G^*$ be semisimple such that
$\chi\in\cE(G,s)$ and set $\bC:=\Cent_{\bG^*}(s)$, $C:=\bC^F$ and
$C^\circ:={\bC^\circ}^F$. Let $\Psi(\chi)\in\cE(C,1)$ denote the unipotent
Jordan correspondent of $\chi$. Then by Lusztig's Jordan decomposition
$\chi(1)=|G^*:C|_{p'}\,\Psi(\chi)(1)$ (see \eqref{eq:Jordan dec}),
so both $|G^*:C|$ and $\Psi(\chi)(1)$ have to be odd. Thus $\Psi(\chi)$
lies above a unipotent character of ${\bC^\circ}^F$ of odd degree, and hence
in the principal series of $C$ by Proposition~\ref{prop:unip}. So its degree
polynomial is prime to $X-1$ by Lemma~\ref{lem:deguni} and~\eqref{eq:HLdeg}.
Hence the order polynomial of $|G^*:C|$ must
be divisible by $X-1$ by our assumption. On the other hand, as $|G^*:C|$ is
odd, $C$ contains a Sylow 2-subgroup of $G^*$. \par
If $q\equiv1\pmod4$ then we may argue as follows. By \cite[Thm.~5.9]{MaH0},
$\bC$ must contain a Sylow 1-torus of $\bG^*$. But then the order polynomial
of $|G^*:C|$ cannot be divisible by $X-1$, a contradiction. \par
So now assume that $q\equiv3\pmod4$. Then again by \cite[Thm.~5.9]{MaH0},
$\bC$ must contain a Sylow 2-torus of $\bG^*$. The order $|C|$ is given by a
polynomial in $q$ of the form $cf(q)$, where $c=|C:C^\circ|$ and
$f\in\ZZ[X]$ is monic. Note that $\bC/\bC^\circ$ is isomorphic to a subgroup
of the fundamental group of $\bG$, hence of the center of $\bG$ (see
\cite[Prop.~14.20]{MT}). In particular, as $\bG$ is simple and not of type
$\tA_l$ we have $|C:C^\circ|_2\le4$, and in fact $|C:C^\circ|_2\le2$
unless $\bG$ is of type $\tD_l$. As $X-1$ divides $f$, we are done if either
$G$ has odd order center, or if $\bC$ is connected.
\par
So $\bG$ is of type $\tB_l,\tC_l,\tD_l$ or $\tE_7$. For $\bG$ not of type
$\tD_l$ with $l$ odd we know by Lemma~\ref{lem:centSyl} applied to
$\bH:=\bG^*$ that
$s$ must be an involution. For $\bG$ of type $\tE_7$, the 2-central involutions
of $G^*$ have centraliser of type $\tD_6(q)\tA_1(q)$, whose order polynomial
is divisible by the full power $(X-1)^7$ of $X-1$ occurring in the polynomial
order of $G^*$, contrary to what we showed. For the classical type groups, let
us first observe that $\bC$ cannot be contained inside a proper $F$-stable
Levi subgroup $\bL$ of $\bG^*$, because $L=\bL^F$ has even index in $G^*$.
Indeed, a Sylow 2-subgroup of $G^*$, or $L$, is contained in the normaliser of
a Sylow 2-torus of $\bG^*$ (see \cite[Cor.~25.17]{MT}), respectively of $\bL$.
But this normaliser is an extension of that Sylow 2-torus by the Weyl group
of $G^*$, $L$ respectively. The claim then follows since any proper parabolic
subgroup of a Weyl group $W$ of type $\tB_l$ or $\tD_l$ has even index in $W$.
\par
Thus, $s$ is quasi-isolated in $\bG^*$, and hence occurs in
Table~\ref{tab:2-central}. For $\bG$ of type $\tB_l$ the dual group is
of adjoint type $\tC_l$, and there the listed centraliser does contain a
Sylow 1-torus. For $\bG$ of type $\tC_l$ and so $\bG^*$ of type $\tB_l$ the
listed centralisers either contain a Sylow 1-torus or are given in the
statement with $l$ odd. As the order polynomial of $\bC^\circ$ is
divisible by $(X-1)^{l-1}$ in these cases, the degree polynomial of $\chi$
is divisible by $X-1$ just once, so by Lemma~\ref{lem:degHCseries}, $\chi$ lies
in the Harish-Chandra series of a cuspidal character of a Levi subgroup $L$
of $G$ of rank 1, hence a Levi subgroup of type~$\tA_1$. Now $G$ has two
conjugacy classes of such Levi subgroups, one with connected center lying in
the stabiliser $\GL_l(q)$ of a maximally isotropic subspace, the other
isomorphic to $\Sp_2(q)\times(q-1)^{l-1}$. The degrees of cuspidal
characters of these two types of subgroups are $q-1$, and also
$\frac{1}{2}(q-1)$ for the second type. The latter ones are thus the only ones
of odd degree. An easy variation of the proof of Proposition~\ref{prop:unip},
using that $X-1$ is prime to $X+1$ now shows that if $\chi$ has odd degree,
it must lie above the cuspidal characters of $\Sp_2(q)\times(q-1)^{l-1}$ of
degree $\frac{1}{2}(q-1)$.
\par
So finally assume that $\bG$ is of type $\tD_l$. Recall that $\bC^\circ$ cannot
contain a Sylow 1-torus of $\bG^*$. The only centralisers in
Table~\ref{tab:2-central} not containing a Sylow 1-torus of the ambient
group are $\tD_k(q)\cdot\tw2\tD_{l-k}(q).2$ with $l$ even and $k$ odd inside
$\tw2\tD_l(q)$. But these do not contain a Sylow 2-torus, contrary to what we
know has to happen. So we get no example in type $\tD_l$.
\end{proof}

\begin{rem}
The precise conditions on $k$ for $\tB_{2k}(q)\cdot\tw2\tD_{l-2k}(q)$ in
Theorem \ref{thm:odd degree} to contain a Sylow 2-subgroup of
$G^*=\SO_{2l+1}(q)$ are worked out in \cite[Prop.~4.2]{MaExt}. In fact,
in \cite[Thms.~4.10 and 4.11]{MaExt} it is shown that $G=\Sp_{2l}(q)$ satisfies
the inductive McKay condition for the prime~2 if $q$ is an odd power of~$p$.
\end{rem}

The following consequence will be used in the proof of
Proposition~\ref{prop:9_2}:

\begin{lem}   \label{lem:label odd chars}
 Let $\bG$ be simple, of simply connected type, not of type $\tA_l$ or $\tC_l$.
 Let $\chi\in\Irr_{2'}(G)$. Then $\chi=\R_T^G(\lambda)_\eta$, where $T$ is a
 maximally split torus of $G$, $\lambda\in \Irr(T)$ is such that
 $2\nmid |W:W(\lambda)|$ and $\eta\in\Irr(W(\lambda))$ is of odd degree.
\end{lem}

\begin{proof}
By Theorem~\ref{thm:odd degree} every $\chi\in\Irr_{2'}(G)$ occurs
in the principal series, that is, it lies in the Harish-Chandra series of a
linear character $\la\in\Irr(T)$. According to our remarks before
Lemma~\ref{lem:deguni} we have that
$\chi(1)=|G:T|_{p'}\,D_\chi(q)\,\la(1)$, where $\la(1)=1$. Write
$f_\chi'\in\QQ[X]$ for the degree polynomial of $\chi$. Since
$\Phi_i(q)_2\ge\Phi_i(1)_2$ for all $i\ge2$ it follows that $f_\chi(1)$ is
odd. Now replacing $|G:T|_{p'}$ by the order polynomial and then specialising
at $q=1$ we obtain that $|W|\,D_\chi(1)=|W|\,\eta(1)/|W(\la)|$ is odd, for
$\eta\in\Irr(W(\la))$ the label of $\chi$,
whence the two integers $\eta(1)$ and $|W:W(\la)|$ are odd.
\end{proof}

\subsection {Proof of Theorem~\ref{thm:McKayp=2}}
We now combine the above results on Harish-Chandra induction and on characters
of odd degree to complete the proof that every simple group
satisfies the inductive McKay condition for $\ell=2$, and thus
Theorem~\ref{thm:McKayp=2} holds.

We have seen in Theorem~\ref{thm:8_1} that our results are sufficient to prove
that the inductive McKay condition holds for the prime~$2$,
whenever $4|(q-1)$ and $\Phi$ is of type $\tB_l$, $\tC_l$ or $\tE_7$. The
result even applies to the simple groups which are the quotients of
$\tw 2\tD_l(q)$ or $\tw 2 \tE_6(q)$. Taking into account earlier results
summarised in Proposition~\ref{prop:exc} the only cases that are left to
consider are the simple groups associated with $\tD_{l,\SC}(q)$ and
$\tE_{6,\SC}(q)$.

The following specific considerations are  tailored to the case where $\ell=2$.

\begin{prop}   \label{prop:9_2}
 Let $G=\bG^F$ be as above. Let $\chi\in \Irr_{2'}(G)$. Then $\chi$ extends
 to its inertia group in $GD$. Thus the assumption~\ref{2_2gloext} holds.
\end{prop}

\begin{proof}
The statement is trivial whenever $D_\chi$ is cyclic.
If $G\not \cong \tD_{4,\SC}(q)$ the Sylow $r$-subgroups of $D$ are cyclic
for any odd prime $r$. Note that $\det \chi$ is trivial since $G$ is
perfect. Then \cite[Thm.~6.25]{Isa} shows that $\chi$ extends to its
inertia group in $G D_2$. According to \cite[(11.31)]{Isa} this implies
that $\chi$ extends to $G D_\chi$.

It remains to consider the case where $G \cong \tD_{4,\SC}(q)$.
Following the considerations above we have to show that $\chi$ extends to its
inertia group in $G D_3$ for any Sylow $3$-subgroup $D_3$ of $D$. Assume
that $D_\chi$ has a non-cyclic Sylow $3$-subgroup.

According to Theorem~\ref{thm:odd degree} there exists a character
$\lambda\in\Irr(T)$, where $T$ is a maximally split maximal torus of $G$,
such that $\chi$ is a constituent of $\R_T^G(\lambda)$. According to
Lemma~\ref{lem:label odd chars}, $\chi$ corresponds to some character
$\eta\in\Irr(W(\lambda))$ of odd degree such that $\chi$ has
multiplicity $\eta(1)$ in $\R_T^G(\lambda)$. Moreover $2\nmid |W:W(\lambda)|$.

Let $\gamma \in D$, $F'\in \spann<F_0>$ be such that $\langle\gamma,F'\rangle$ is
the Sylow $3$-subgroup of $D_\chi$. Direct computations show that any Sylow
$2$-subgroup of $W$ is self-normalising in $W$ and can be chosen to be
$\gamma$-stable. Let $P$ be such a $\gamma$-stable Sylow $2$-subgroup of $W$.
Then after some $N$-conjugation of $\la$ we can assume that $W(\lambda)$
contains~$P$.

Then there exist elements $n,n'\in N$ such that $n\gamma$ and $n'F'$ stabilise
$\lambda$ and $P$. Since $P$ is  $\langle \gamma, F'\rangle$-invariant
the elements $\pi(n),\pi(n')$ are contained in $\NNN_W(P)=P$.
As linear character $\lambda$ has an extension $\w\lambda$ to
$\langle T,F',\gamma\rangle$. Since $\gamma$ and $F'$ stabilise the unipotent
radical $U$ of the Borel subgroup $B$ of $G$, $\w\lambda$ lifts to a character
$\hat\la\in \Irr(\spann<B,F',\gamma>)$. The induced character
$\Gamma=\hat\la^{\langle G,F',\gamma\rangle}$ is then an extension of the
character $\R_T^G(\lambda)$. Hence $\restr\Gamma|G$ has $\chi$ as constituent
with odd multiplicity $\eta(1)$.

Since $|W:W(\la)|$ is odd and $P$ has index~$3$ in $W$, we either have
$W(\lambda)=W$ or $W(\lambda)=P$. In the first
case, $\la=1$ since $W$ acts fixed point freely on $\Irr(T)$ and so
$\chi$ is unipotent, in which case the statement is an easy consequence of
\cite[Thm.~2.4]{MaExt}. Else, the character $\eta$ of $W(\lambda)$ of odd
degree is linear since $W(\lambda)$ is a $2$-group, and so $\chi$ has
multiplicity one in $\R_T^G(\lambda)$. Hence $\Gamma$ has a unique constituent
$\w\chi$ that is an extension of $\chi$. This completes the proof.
\end{proof}

Together with Theorem~\ref{thm:8_1} and Lemma~\ref{lem:excSchur}
this completes the proof of Theorem~\ref{thm:d=1good}.

\begin{proof}[Proof of Theorem \ref{thm:McKayp=2}]
By \cite[Thm.~B]{IMN} it is sufficient to show that all non-abelian simple
groups $S$ satisfy the inductive McKay condition. For the simple groups not of
Lie type this is known by \cite[Thm.~1.1]{ManonLie}. So now assume that $S$ is
of Lie type, and not as in Proposition~\ref{prop:exc}. Observe that $\ell=2$
implies that $d_\ell(q)\in\{1,2\}$. So it suffices to verify the assumptions
in Theorem~\ref{thm:8_1}(c) and~(d). Condition~\ref{2_2glostar} holds since
all $2'$-character only lie in very specific Harish-Chandra series thanks to
Theorem~\ref{thm:odd degree} and characters in those series, more precisely
the structure of their stabilisers, have been studied in
Corollary~\ref{cor:7_3}. The requirement \ref{2_2gloext} is satisfied for
characters in $\Irr_{2'}(G)$ thanks to Proposition~\ref{prop:9_2}.
\end{proof}


\end{document}